\documentclass[12pt]{article}

\usepackage{ifthen}
\usepackage[latin1]{inputenc}
\usepackage{amsmath,amssymb,amsthm}

\usepackage{graphicx}

\newcommand{\hookdownarrow}{\mathrel{\rotatebox[origin=c]{-90}{$\hookrightarrow$}}}

\newcommand{\A}{\mathbb{A}}
\def\AA{{\mathbb A}}

\def\CC{{\mathbb C}}

\def\PP{{\mathbb P}}

\let\goth=\mathfrak

\newcommand{\red}{_{gr}}

\newcommand{\f}{\frac}
\newcommand{\findem}{\nolinebreak\vspace{\baselineskip} \hfill\rule{2mm}{2mm}}
\renewcommand{\phi}{\ensuremath{\varphi}}

\newcommand{\pla}{\A^2}

\renewcommand{\phi}{\ensuremath{\varphi}}

\newcommand{\codim}{\mbox{codim}\;}

\newtheorem{nt}{Notation}[section]

\newtheorem{prop}[nt]{Proposition}

\newtheorem{exercice}{Exercice}
\newcounter{numeroquestion} 
\newcounter{numerosousquestion}

\newcommand{\sousquestion}{\ifthenelse{\value{numerosousquestion}=1}{}{\\}\textbf{\roman{numerosousquestion})} \addtocounter{numerosousquestion}{1}}
\newtheorem{correction}{Correction}

\newenvironment{listecompacte}
{\begin{list}
    {\ensuremath{\bullet}}
    {\setlength{\topsep}{2pt}
      \setlength{\itemsep}{1pt} \setlength{\parsep}{0pt}}
}
{\end{list}
}

\newtheorem{coro}[nt]{Corollary} 
\newtheorem{defi}[nt]{Definition}

\newtheorem{ex}[nt]{Example} 
\newtheorem{lm}[nt]{Lemma} 
 
\newtheorem{rem}[nt]{Remark} 
\newtheorem{thm}[nt]{Theorem}

\usepackage[all]{xy}
\usepackage{url}
\usepackage{array}
\usepackage{hyperref}

\newtheorem*{mercis}{Acknowledgments}

\newcommand{\N}{\mathbb N}
\newcommand{\gl}{\mathfrak {gl}}
\newcommand{\g}{\mathfrak g}
\newcommand{\pp}{\mathfrak p}
\newcommand{\bb}{\mathfrak b}
\newcommand{\qq}{\mathfrak q}
\newcommand{\nf}{\mathfrak n}
\newcommand{\wfr}{\mathfrak w}

\newcommand{\K}{{\Bbbk}}

\newcommand{\MCM}{\mathcal M}
\newcommand{\MN}{\mathcal N}
\newcommand{\MNTilde}{\widetilde{\mathcal N}}

\newcommand{\Pn}{\mathcal P}
\newcommand{\bolda}{\boldsymbol{\lambda}}
\newcommand{\boldmu}{\boldsymbol{\mu}}

\newcommand{\sn}{S^{[n]}}
\newcommand{\sno}{S^{[n]}_0}
\newcommand{\zno}{Z^{[n]}_0}
\newcommand{\zoemb}{Z^{[n_j,n_{j-1},\dots,n_1]}_0}
\newcommand{\sxo}[1]{S^{[#1]}_0}
\newcommand{\sxn}[1]{S^{[#1,n]}}
\newcommand{\skn}{\sxn k}
\newcommand{\sxno}[1]{S^{[#1,n]}_0}
\newcommand{\sxyo}[2]{S^{[#1,#2]}_0}
\newcommand{\skno}{\sxno k}
\newcommand{\sxnemb}[1]{S^{[\![#1,n]\!]}}
\newcommand{\sknemb}{\sxnemb k}
\newcommand{\sxnoemb}[1]{S^{[\![#1,n]\!]}_0}
\newcommand{\sknoemb}{\sxnoemb k}
\newcommand{\sdeuxno}{\sxno 2}

\newcommand{\decp}[3]{
\left(\begin{array}{c| c c c}#1& & #2& \\ \hline 0 & & &\\ \vdots & & #3 &\\0 & & & \end{array}\right)} 
\newcommand{\decpp}[3]{
\left(\begin{array}{c c c c}#1& & #2& \\  0 & & &\\ \vdots & & #3 &\\0 & & & \end{array}\right)} 
\newcommand{\func}[4]{\left\{\begin{array}{r c l}#1&\rightarrow & #2\\ #3&\mapsto & #4\end{array}\right.}

\newcommand{\alter}[4]{\left\{\begin{array}{l l}#1&\mbox{#2} \\#3&\mbox{#4}\end{array}\right.}
\newcommand{\pxlnil}{(\pp^{X_{\bolda}})^{nil}}

\newcommand\undermat[2]{%
  \makebox[0pt][l]{$\smash{\underbrace{\phantom{%
    \begin{matrix}#2\end{matrix}}}_{\text{$#1$}}}$}#2}
    \newlength{\arrayrulewidthOriginal}
    \newcommand{\Hline}[1]{%
      \noalign{\global\setlength{\arrayrulewidthOriginal}{\arrayrulewidth}}%
      \noalign{\global\setlength{\arrayrulewidth}{#1}}\hline%
      \noalign{\global\setlength{\arrayrulewidth}{\arrayrulewidthOriginal}}}
    
\DeclareMathOperator{\Hom}{Hom}
\DeclareMathOperator{\End}{End}
\DeclareMathOperator{\Ker}{Ker}
\DeclareMathOperator{\Img}{Im}
\DeclareMathOperator{\GL}{GL}
\DeclareMathOperator{\Lie}{Lie}

\begin{document}
\sloppy
\title{Nested punctual Hilbert schemes and commuting varieties of
  parabolic subalgebras\thanks{This work has benefited from two short stay funded by the GDR TLAG}}
\date{\today}
\author{{\sc Micha\"el Bulois}\thanks{Universit\'e de Lyon, CNRS UMR 5208, Universit\'e Jean Monnet, Institut Camille Jordan, Maison de l'Universit\'e, 10 rue Tr\'efilerie, CS 82301, 42023 Saint-Etienne Cedex 2, France. \href{mailto:michael.bulois@univ-st-etienne.fr}{michael.bulois@univ-st-etienne.fr}} \phantom{aaa} 
{\sc Laurent Evain}\thanks{Universit\'e d'Angers,
Facult\'e des Sciences,
D\'epartement de maths,
2 Boulevard Lavoisier,
49045 Angers Cedex 01,
FRANCE. \href{mailto:laurent.evain@univ-angers.fr}{laurent.evain@univ-angers.fr}}}
\maketitle


\section*{Abstract: }
It is known that the variety parametrizing pairs of commuting
nilpotent matrices is irreducible and that this provides a proof of
the irreducibility of the punctual Hilbert scheme in the plane. 
We extend this link to the nilpotent commuting
variety of some parabolic subalgebras of $M_n(\K)$ and to the punctual nested Hilbert scheme. By this method, we obtain
a lower bound on the dimension of these moduli spaces. We characterize the
cases where they are irreducible. In some
reducible cases, we describe the
irreducible components and their dimensions. 


\tableofcontents
\section{Introduction}
Let $\K$ be an algebraically closed field of arbitrary characteristic. 

Let $S^{[n]}$ denote the Hilbert scheme parametrizing the 
zero dimensional schemes $z_n$ in the affine plane $S=\AA^2=Spec\ \K[x,y]$
with $length(z_n)=n$. Several variations from this original 
Hilbert scheme
have been considered. For instance, Brian\c con 
studied the punctual Hilbert scheme $\sno$ 
which parametrizes the subschemes $z_n$ with length $n$ and support 
on the origin \cite{Br}, and Cheah has considered 
the nested Hilbert schemes parametrizing tuples of zero dimensional
schemes $z_{k_1}\subset z_{k_2}\subset \dots \subset z_{k_r}$
organised in a tower of successive inclusions
\cite{Ch1,Ch2}.

Let ${\cal C}(M_n)$ be the commuting variety
of $M_n$, \emph{i.e.} the variety parametrizing the pairs 
of square matrices $(X,Y)$ with
$X\in M_n(\K),Y\in M_n(\K)$, $XY=YX$. 
Gerstenhaber \cite{Ge} proved the irreducibility 
of ${\cal C}(M_n)$. 
Many variations in the same circle of ideas have been considered.
For instance, one can consider ${\cal C}(\goth a)$,
where  $\goth a\subset M_n$ is a subspace (often a Lie subalgebra),
or $\MN(\goth a)\subset
{\cal C}(\goth a)$ defined by the condition that $X,Y$ be nilpotent  (cf. \emph{e.g.} \cite{Pa, Bar, Pr,  Bu, GR}).

There is a well known connection between Hilbert schemes and commuting varieties. 
If $z_n\in S^{[n]}$ is a zero dimensional subscheme, and if 
$b_1,\dots,b_n$ is a base of the structural sheaf ${\cal O}_{z_n}= \K[x,y]/I_{z_n}$,
the multiplications by $x$ and $y$ on  ${\cal O}_{z_n}$ are represented 
by a pair of commuting matrices $X,Y$. The scheme $z_n$ is
characterized by the pair of commuting matrices $(X,Y)$ up to simultaneous conjugation.
This link has been intensively 
used by Nakajima \cite{Na}. Obviously, variations 
on the Hilbert scheme correspond to variations on the commuting
varieties.

The goal of this paper is to study the 
punctual nested Hilbert schemes 
$\skno$ and $\sknoemb$ and their matrix 
counterparts $\MN(\pp_{k,n})$ and $\MN(\qq_{k,n})$. Here 
$S_0^{[k,n]}\subset S_0^{[k]}\times S_0^{[n]}$ parametrizes the pairs 
of punctual schemes $z_k,z_n$ with $z_k\subset z_n$ and $\sknoemb \subset
S_0^{[k]}\times S_0^{[k+1]} \times\dots\times  S_0^{[n]}$ 
parametrizes the tuples $z_k\subset z_{k+1}\dots \subset z_n$,
$\pp_{k,n}\subset M_n$ is a parabolic subalgebra defined by 
a flag $F_0\subset F_k \subset F_n$ with $\dim F_i=i$ and  
$\qq_{k,n}$ is associated with a flag $F_0\subset F_1\dots 
\subset F_k\subset F_n$. 

Our interest in the nested punctual Hilbert schemes 
stems from the the creation and
annihilation operators on the cohomology of the Hilbert scheme 
introduced by Nakajima and Grojnowski \cite{Na, Gr}.
The geometry of the nested Hilbert schemes controls these operators.
A typical application is the vanishing of a cohomology class 
which is the push-down of the class of a variety under a 
contracting morphism. It is often necessary
to describe the components of the nested Hilbert schemes 
and/or their dimension
to simplify the computations \cite{Na,Le,CE}.
On the Lie algebra side, the subalgebras $\pp_{k,n}\subset M_n$
are the maximal parabolics, hence are prototypes for the study of
general parabolics. On the other hand, the algebras $\qq_{k,n}$ are
used as a tool to study some other cases and are well behaved for our
computations..  
Closely linked to this setting, note also that $\qq_{n,n}$ is a Borel
subalgebra of $M_n$. Some properties of $\MN(\mathfrak
\qq_{n,n})$ can be found in \cite{GR}.

Let  $P_{k,n}$, resp.
$Q_{k,n}$, be the groups of invertible matrices in $\pp_{k,n}$, resp. $\qq_{k,n}$. It acts on $\pp_{k,n}$, resp. $\qq_{k,n}$, by conjugation. In the Lie algebra setting, $P_{k,n}$, resp. $Q_{k,n}$, is nothing but the parabolic subgroup of $\GL_n(\K)$ with  Lie algebra $\pp_{k,n}$, resp. $\qq_{k,n}$.

It is possible 
in our context  
to make precise the connection between  Hilbert schemes and 
commuting varieties.
Since zero dimensional schemes are characterized by pairs of commuting
matrices up to the choice of the base, the expectation is that 
Hilbert schemes should be quotients of commuting varieties. 
This is correct in essence, provided that one takes care of the
existence of cyclic vectors. Moreover,  
the acting groups $P_{k,n}$ and
$Q_{k,n}$ are not reductive. Nevertheless, 
we will construct a geometric quotient in the sense of Mumford
\cite{Mu}, as follows. 

Let $ \MNTilde^{cyc}(\pp_{k,n})$ and $\MNTilde^{cyc}(\qq_{k,n})$ be the open
  loci in  $ \MN(\pp_{k,n})\times \K^n$ and $\MN(\qq_{k,n})\times \K^n$ defined by the
  existence of a cyclic vector, \emph{i.e.} these open loci parametrize the
  tuples $((X,Y),v)$ with $\K[X,Y](v)=\K^n$. 
They are stable under the respective action of $P_{k,n}$ and $Q_{k,n}$. 

\newcommand{\quo}{q}
\newcommand{\quoprime}{q'}
\newcommand{\isompHilb}{i}
\newcommand{\isomqHilb}{i'}


\begin{thmSansNumero}{ \ref{thmcorres}}
  \begin{enumerate} 
  \item 
  There exist  geometric quotients $\quo:\MNTilde^{cyc}(\pp_{k,n})\to
  \MNTilde^{cyc}(\pp_{k,n})/P_{k,n}$ and $\quoprime:\MNTilde^{cyc}(\qq_{k,n})\to
  \MNTilde^{cyc}(\qq_{k,n})/Q_{k,n}$  and they are principal bundles locally
  trivial for the Zariski topology.
\item There exist surjective morphisms
  $\widetilde{\pi}_{k,n}:\begin{array}{rcl}\MNTilde^{cyc}(\pp_{k,n})&\rightarrow
      &\sxno{n-k}\end{array}$,
$\widetilde{\pi}'_{k,n}:\begin{array}{rcl}\MNTilde^{cyc}(\qq_{k,n})&\rightarrow
    &\sxnoemb{n-k}\end{array}$.
\item There exist isomorphisms $\isompHilb:\MNTilde^{cyc}(\pp_{k,n})/P_{k,n} \to
  \sxno{n-k}$ and  $\isomqHilb:\MNTilde^{cyc}(\qq_{k,n})/Q_{k,n} \to
  \sxnoemb{n-k}$. These isomorphisms identify the projections to 
the Hilbert schemes with the geometric quotients, \emph{i.e.}
  $\isompHilb \circ \quo=\widetilde{\pi}_{k,n}$ and  $\isomqHilb \circ \quoprime=\widetilde{\pi}'_{k,n}$.
\end{enumerate}
\end{thmSansNumero}

This is directly inspired from the general construction of Nakajima's quiver varieties (see \emph{e.g.} \cite{Gi}), the cyclicity being a stability condition in the sense of \cite{Mu}. It can straightforwardly be generalized to any parabolic subalgebra of $M_n$.

We then investigate the dimension and the number of components of  
$ \MN(\pp_{k,n})$, $\MN(\qq_{k,n})$, $\sxno{k}$ and $\sxnoemb{k}$.  
Many of our proofs consider the problem for $
\MN(\pp_{k,n})$, $\MN(\qq_{k,n})$ firstly and then use
the above theorem and some geometric arguments to push down 
the information to the Hilbert schemes.  Conversely, sometimes, we pull back the
information from the Hilbert scheme to the commuting variety.  
The
general philosophy is that the problems on the
commuting varieties are in some sense ``linear'' versions 
of the corresponding problems on the Hilbert scheme 
which are ``polynomial'' problems. This explains why the most frequent direction of
propagation of the information is from commuting varieties to Hilbert
schemes.

\begin{thmSansNumero}{ \ref{thmCNSIrred}}
  $\skno$ is irreducible if and only if $k\in \{0,1,n-1,n\}$. The
  variety $\MN(\pp_{k,n})$ is irreducible if and only if $k\in \{0,1,n-1,n\}$. 
\end{thmSansNumero}


\begin{thmSansNumero}{ \ref{thmCNSIrredq}}
$\sknoemb$ is irreducible if and only if $k\in \{n-1,n\}$ or $n\leqslant 3$. 
  $\MN(\qq_{k,n})$ is irreducible if and only if $k\in \{0,1\}$ or $n\leqslant 3$.
\end{thmSansNumero}
\medskip

When $k=2$ or $k=n-2$, we have precise results on the number 
of components and their dimensions. 

\begin{thmSansNumero}{ \ref{propp2n}}
Let $\wfr=\qq_{2,n}$ or $\pp_{2,n}$. Then $\MN(\wfr)$ is
equidimensional of dimension $\dim \wfr-1$. It
 has $\left\lfloor \frac n 2 \right\rfloor$ components. 
\end{thmSansNumero}

\begin{thmSansNumero}{ \ref{corpropp2n}}
  $\sxno{2}$, $\sxno{n-2}$, $\sxnoemb{n-2}$ are equidimensional of dimension
  $n-1$. They have $\left \lfloor \frac{n}{2} \right\rfloor$ components. 
\end{thmSansNumero}

\medskip

The similarity between $\sxno{k}$ and $\sxno{n-k}$ follows from a
transposition isomorphism between $\MN(\pp_{k,n})$ and
$\MN(\pp_{n-k,n})$. 
Note however that there might be profound differences between
the Hilbert schemes
and the corresponding commuting varieties because of
the cyclicity condition, see remark \ref{remSurLaSymetrie}. 

Without any assumption on $k\in [\![0,n]\!]$, we have an estimate for the dimension
of the components.

\begin{propSansNumero}{ (Section \ref{seclowbound})}
  Each  irreducible component of $\sxnoemb{k}$ has dimension
  at least $n-1$ which is the dimension of the
  curvilinear component. Each irreducible component
  of $\sxno{k}$ has dimension at least $n-2$, which is the
  dimension of the curvilinear component minus one. Each irreducible
  component of $\MN(\qq_{k,n})$ has dimension at least
  $\dim\qq_{k,n}-1$.  Each irreducible component of $\MN(\pp_{k,n})$
  has dimension at least $\dim\pp_{k,n}-2$.  
\end{propSansNumero}

Note that the result is not optimal for $\pp_{k,n}$ and $\sxno{k}$ as
Theorems \ref{propp2n} and \ref{corpropp2n} show. 

Our approach does not depend on the characteristic of $\K$. One
reason that makes this possible is that we often rely on the key work
of Premet in \cite{Pr} made in arbitrary characteristic.

Several statements in the paper allow generalisations or abstract
reformulations. To keep the paper readable by a large audience, 
we have chosen a presentation which minimizes the prerequisites.
Hopefully, the paper is readable by 
a non specialist in at least one of the domains Hilbert
schemes/commuting varieties.

\begin{mercis}
We are grateful to Markus Reineke for computing and communicating to us the example mentioned in Remark~\ref{remfin}. 
\end{mercis}
\section{Reducible nested Hilbert schemes}

\label{sec:reduc-nest-hilb}
Throughout the paper, we work over an algebraically closed field $\K$
of arbitrary characteristic.

In this section, we produce examples of reducible nested Hilbert
schemes, and we identify some of their components via direct computations. 
 

Let $S=\pla=Spec\;\K[x,y]$ be the affine plane. We denote by $S^{[n]}$
the Hilbert scheme parametrizing the zero dimensional subschemes $z_n\subset \pla$ 
of length $n$. We denote by $\skn \subset S^{[k]}\times S^{[n]}$ the Hilbert
scheme parametrizing the pairs $(z_k,z_n)$ with $z_k\subset
z_n$. We denote by $\sknemb \subset
S^{[k]}\times S^{[k+1]} \times\dots\times  S^{[n]}$ 
the Hilbert scheme that parametrizes the tuples of subschemes $(z_k,z_{k+1},\dots,z_n)$ 
with $z_k\subset z_{k+1}\dots \subset z_n$. An index $0$ indicates
that the schemes considered are supported on the origin. For instance, 
$\skno\subset S^{[k]}_0 \times S^{[n]}_0$ is the Hilbert
scheme parametrizing the pairs $(z_k,z_n)$ with $z_k\subset
z_n$ and $supp(z_k)=supp(z_n)=O$. 

All these Hilbert schemes have a functorial description. 
For the original Hilbert scheme, see \cite{Gro} or \cite{HM} for a
modern treatment. For the nested Hilbert schemes see \cite{Kee}.
For the versions supported on the origin, a good reference is \cite{Ber}.
Section \ref{sec:descriptionsFonctorielles} will recall the main technical descriptions that we
need. 


\begin{prop} \label{irredOnlyFor1nMoins1}
  For $k\neq 0,1,n-1,n$, $\skno$ is reducible.
\end{prop}
\begin{proof}
  Recall that a curvilinear scheme of length $n$ 
  is a punctual scheme which can be defined by the ideal
  $(x,y^n)$ in some system of coordinates \emph{i.e.} this is a punctual
  scheme included in a smooth curve. The curvilinear schemes  form an
  irreducible subvariety of 
  $S_0^{[n]}$ of dimension $n-1$ \cite{Br}.
  We prove that $\skno$ admits at least two components: the
  curvilinear component where $z_k$ and $z_n$ are both curvilinear (of
  dimension $n-1$ since $z_k=(x,y^k)$ is determined by $z_n=(x,y^n)$ ) 
  and an other component of dimension  greater or equal than $n-1$. 
  The families that we exhibit below are special cases of more general
  constructions which 
  give charts on the Hilbert schemes \cite{Ev}.

Consider the families
of subschemes $z_k$, $z_n$, with equation $I_k$ and
$I_n$ where $I_n=(x^{n-1},yx+\sum_{i=2}^{n-2} a_i x^i,
y^2+\sum_{i=2}^{n-2} a_i yx^{i-1}+bx^{n-2})$. Let $\phi$ be the
change of coordinates defined by $x\mapsto x$, $y\mapsto
y-\sum_{i=2}^{n-2} a_i x^{i-1}$. Then
$\phi(I_n)=(x^{n-1},yx,y^2+bx^{n-2})$. In particular, for each choice
of the 
parameters $a_i,b$, the scheme $z_n$ has length $n$.  


We may suppose $n\geq 4$, otherwise there are no integers $k$ to
consider in the proposition. Then  all the generators of $I_n$ 
have valuation at least two and it follows that $z_n$ is not curvilinear.

For each $z_n$, there is a one dimensional family
of subschemes $z_k\subset z_n$. We check this claim in the coordinate
system where $I_n=(x^{n-1},yx,y^2+bx^{n-2})$.
Consider $I_k=(x^{k},y-cx^{k-1})$. Modulo $I_k$ we have $x^{n-1}=0$
and $yx=cx^k=0$. Since $k\leq n-2$ and $k\geq 2$, 
$y^2+bx^{n-2}=y^2=(cx^{k-1})^2=0$.   Thus $I_n\subset I_k$, as
expected. 

All the ideals $I_n$ and $I_k$ are pairwise distinct since their generators form a
reduced Gr\"obner basis for the order $y>>x$
and a reduced Gr\"obner
basis is unique (\cite{Eis}, Exercise 15.14). We thus 
have two families of dimension $n-1$, namely the curvilinear
component and the family we constructed with the parameters
$(a_i,b,c)$. It remains to prove that
they cannot be both included in a same component $V$ of dimension
$\geq n$.  For this, we prove that the closure of the 
curvilinear locus is an irreducible component. 

Let $p$ be the projection $\skno\rightarrow \sno$. Let
$C^n\subset   \sno$ be the curvilinear locus and
$C^{k,n}=(p^{-1}(C^n))_{red}$ be the reduced inverse image.
Note that $p$ restricts to a bijection between $C^{k,n}$ and $C^n$. Let $V$ be an
irreducible variety containing the curvilinear locus $C^{k,n}$. 
Since $C^n$ is open in $p(V)\subset \sno$ by \cite{Br} and since $p$ restricts to a
bijection between $C^{k,n}$ and $C^n$, we have $\dim V=\dim C^n=n-1$. 
\end{proof}

In general $\skno$ has more than the 
two components exhibited in Proposition \ref{irredOnlyFor1nMoins1}. 
For instance, corollary \ref{corpropp2n} shows that 
 $\sdeuxno$ is equidimensional with  $\left \lfloor \frac{n}{2} \right \rfloor$ components.
As a first step towards this goal, we count the number of components
of dimension $n-1$. 


\begin{prop}\label{composantesDeDimNMoins1DansSdeuxno}
$\sdeuxno$ contains exactly $\lfloor n/2 \rfloor $ components of dimension $n-1$.   
\end{prop}
\begin{proof}
Consider the action of the torus $t.x=t^{k}x$ ($k>>0$) , $t.y=ty$
on
$\K[x,y]$ and the induced action on $\sno$. 
There is a Bialynicki-Birula decomposition of $\sno$ with respect to
this action. According to \cite[Proof of Proposition 4.2]{ES}, any cell is characterized by a 
partition of $n$, and the dimension of the cell with 
partition $\bolda=(\lambda_1\geqslant \dots\geqslant
\lambda_{d_{\bolda}})$ is $n-\lambda_1$. 
 
There is a
unique cell of dimension $n-1$ of $\sno$ and it is associated with the unique
partition $\bolda=(1,1,\dots,1)$ of $n$ with $\lambda_1=1$. Geometrically, this
cell parametrizes the curvilinear subschemes which intersect the
vertical line $y=0$ with multiplicity one. We call it the curvilinear cell 
and we denote it by $F_{curv}$.
There are $\lfloor n/2 \rfloor$ cells $F_{\bolda}\subset \sno$ of dimension 
$n-2$ corresponding to the partitions $\bolda$ with $n$ boxes and $\lambda_1 =2$ :  one has to take $\bolda=\bolda_{a,b}:=(2^a,1^{b-a})$, with
$b\geq a\geq 1$ and $a+b=n$. 

Following $\cite{Ev}$, we may be more explicit 
and describe the charts corresponding to the
Bialynicki-Birula strata. Since $\sxyo{2}{2}\cong\sxo{2}$ is
homeomorphic to $\mathbb{P}^1$,  where $(c:d)\in
\PP^1$ corresponds to the subscheme $z_2\in \sxo{2}$ with ideal $(cx+dy,x^2,y^2)$,
the
proposition is true for $n=2$ and we may suppose $n\geq 3$ .
If $b=a$, the 
Bialynicki-Birula stratum $F_{\bolda_{a,b}}$ is isomorphic to $Spec\ \K[c_{ij}]$ with 
universal ideal $(x^{a},y^2+\sum_{j\in
  \{0,1\},i\in\{1,\dots,a-1\}}c_{ij}x^iy^j)$. If $b>a$, the stratum is 
$Spec \ \K[c_i,d_i,e_i]$ with universal ideal $( x^{b},
yx^a+\sum_{i\in
  \{1,\dots,b-a-1\}}c_ix^{a+i},y^2+\sum_{i\in\{1,\dots,b-a-1\}}c_iyx^i+\sum_{i\in\{1,\dots,a-1\}}d_i(yx^i+\sum_{j\in
  \{1,\dots,b-a-1\}}c_jx^{i+j})+\sum_{i\in
  \{b-a,\dots b-1\}}e_ix^i$)

 There is at most one term of degree one in the generators of the ideal,
which appears when $(b-a=0,c_{10}\neq 0)$ or $(b-a=1,e_1\neq 0)$. 
In these cases, the corresponding point of the 
Bialynicki-Birula cell parametrizes a curvilinear scheme 
and it parametrizes a noncurvilinear scheme if $b-a \geq 2$ or $e_1=0$
or $c_{10}=0$. There are $\lfloor n/2
\rfloor-1$ partitions $\bolda_{a,b}$ with $b-a\geq 2$. 

Consider the projection $p:\sdeuxno \to \sno$ and $z_n\in \sno$. 
The fiber $p^{-1}(z_n)$ is set-theoretically a point if $z_n$ is
curvilinear. If $z_n$ is not curvilinear, the fiber 
is $\sxo{2}$ which is homeomorphic to $\PP^1$.

It follows that $p^{-1}(F_{curv})$ and $p^{-1}(F_{\bolda_{a,b}})$ with
$b-a\geq 2$ are
irreducible varieties of dimension $n-1$. 
There are $\left\lfloor \frac{n}{2} \right\rfloor$ such irreducible
varieties. To prove that their closures are irreducible components, 
note that $\sdeuxno$ is a proper subscheme of the $n$ dimensional
irreducible variety $\sno \times \sxo{2}$. In particular, 
any irreducible closed subvariety of dimension $n-1$ in $\sdeuxno$ is an irreducible component.

It remains to prove that there  are no other components. Let $L$ be a
component with dimension $n-1$. Since $\sxo{2}$ is one-dimensional,
the generic fiber of the projection
$L\rightarrow \sno$ has dimension $0$ or $1$ thus the projection
has dimension at least $n-2$. If the projection has dimension $n-1$,
then the generic point of $L$ maps to the generic point of the
curvilinear component for dimension reasons, and $L$ 
is the curvilinear component $\overline{p^{-1}(F_{curv})}$. 
If the projection has dimension
$n-2$, then the generic point of $L$ maps to  the generic point of a
Bialynicki-Birula cell of dimension $n-2$, $F_{\bolda_{a,b}}$, 
or to a non closed point of 
$F_{curv}$. Since  
the generic fiber has dimension $1$, 
the generic point of $L$ does not map to $F_{curv}$ nor to 
the generic point of  $F_{\bolda_{a,b}}$, $b-a\leq 1$. 
Hence $L$ is included in one of the
components $\overline{p^{-1}(F_{\bolda_{a,b}})}$  constructed above
with $b-a\geq 2$, and the equality follows from the equality of dimensions.
\end{proof}

\begin{rem}\label{remCodim2commeDim2}
  It is possible to prove along the same lines that $\sxno{n-2}$ has
  exactly $\lfloor n/2 \rfloor $ components of dimension $n-1$.  More
  precisely, the universal ideal    $( P_0=x^{b},
P_1=yx^a+\sum_{i\in
  \{1,\dots,b-a-1\}}c_ix^{a+i},P_2=y^2+\sum_{i\in\{1,\dots,b-a-1\}}c_iyx^i+\sum_{i\in\{1,\dots,a-1\}}d_i(yx^i+\sum_{j\in
  \{1,\dots,b-a-1\}}c_jx^{i+j})+\sum_{i\in
  \{b-a,\dots b-1\}}e_ix^i$)  over 
$F_{\bolda_{a,b}}$ with $b-a\geq 2$ as above defines a $n-2$
dimensional family of subschemes $z_n$ of length $n$. For each such
subscheme $z_{n}$, there is a one dimensional family of subschemes
$z_{n-2}(t)$ parametrized by $t$ with $z_{n-2}(t)\subset z_n$. In
coordinates $z_{n-2}(t)$ is defined by the ideal
$(P_0/x,P_1/x+tx^{b-1},P_2)$ which is well defined since $x$ divides
both $P_0$ and $P_1$. 
By the above, the component containing 
the couples $(z_{n-2},z_n)$ has dimension 
dimension $(n-2)+1=n-1$. Adding the curvilinear component, we obtain in this
way the $\lfloor n/2 \rfloor $ components of dimension $n-1$.
\end{rem}
\section{Hilbert schemes and commuting varieties}
\label{corres}
The goal of this section is to make precise the link between Hilbert schemes and commuting varieties in our context. More explicitly, we realize the Hilbert schemes $\sxno{n-k}$ and $\sxnoemb{n-k}$ as 
geometric quotients of the commuting varieties $\MNTilde^{cyc}(\pp_{k,n})$ and
$\MNTilde^{cyc}(\qq_{k,n})$ by the groups $P_{k,n}$ and $Q_{k,n}$
(Theorem \ref{thmcorres}). As a consequence, we point out a precise
connection between irreducible components of $\sxno{n-k}$
(resp. $\sxnoemb{n-k}$) and those of $\MN^{cyc}(\pp_{k,n})$
(resp. $\MN^{cyc}(\qq_{k,n})$) in Proposition \ref{corresirred}. 

We first introduce the notation to handle the commuting varieties.
Let $M_{n,k}$ be the space of $n\times k$ matrices with entries in $\K$ and let $M_n:=M_{n,n}$. The associative algebra $M_n$ will more often be considered as a Lie algebra $\g$ via  $[A,B]:=AB-BA$ and we will be interested in the action by conjugation of $G=\GL_n$ on it ($g\cdot X=gXg^{-1}$).
If $\wfr$ is a Lie subalgebra of $M_n$ and $X\in \wfr$, we denote the centralizer (also called commutant) of $X$ in $\wfr$ by
$$\wfr^X:=\{Y\in \wfr\,|\, [Y,X]=0\}.$$
The set of elements of $\wfr$ which are nilpotent in $M_n$ is denoted by $\wfr^{nil}$. 
We define the nilpotent commuting variety of $\wfr$:
$$\MN(\wfr)=\{(X,Y)\in (\wfr^{nil})^2\; | \; [X,Y]=0\}\subset \wfr\times\wfr.$$
If a subgroup $Q\subset G$ normalizes $\wfr$ then $Q^X$ is the stabilizer of $X\in \wfr$ in $Q$. 
The group $Q$ acts on $\MN(\wfr)$ diagonally ($q\cdot (X,Y)=(q\cdot X,q\cdot Y)$).

\begin{thm}\label{irrcomvar}
If $X^0$ denotes a regular nilpotent element of $M_n$, we have
$$\MN(M_n)=\overline{G\cdot(X^0,(M_n^{X^0})^{nil})}$$
In particular, the variety $\MN(M_n)$ is irreducible of dimension $n^2-1$
\end{thm}

Recall that an element $X\in M_n$ is said to be regular if it has a cyclic vector, \emph{i.e.} an element $v\in \K^n$ such that $\langle X^k(v)|k\in \N\rangle=\K^n$. This easily implies, and is in fact equivalent to, $\dim G^X(=\dim M_n^X)=n$. There is only one regular nilpotent orbit. This is the orbit of nilpotent elements having only one Jordan block.

This theorem was first stated in \cite{Bar} using the correspondence
with Hilbert schemes (with a small correction in the proof of lemma 3, see MathReviews 1825165). 
We can find other proofs of this theorem in \cite{Bas03} and \cite{Pr}. In \cite{Pr}, the result is proved whithout any assumption on $\textrm{char } \K$.

Let $V=\K^n$ and $(e_1, \dots, e_n)$ be its canonical basis.  We will identify $M_n$ with $\gl(V)$, the set of endomorphisms
of $V$, thanks to this
basis. For $1\leqslant i\leqslant n$, let $V_i=\langle e_1,\dots
e_i\rangle$.
We define $\pp_{k,n}$ (resp. $\qq_{k,n}$)  as the set of matrices
$X\in\gl(V)$ such that  $X(V_k)\subseteq V_k$ (resp. $X(V_i)\subseteq V_i$
for all $1\leqslant i\leqslant k$). 
Given $X\in\pp_{k,n}$, we denote by $X^{(k)}$ the linear map induced by $X$ on $V/V_k$. 
Let $P_{k,n}\subset\GL_n$ (resp. $Q_{k,n}\subset \GL_n$) be the
algebraic group of invertible matrices of $\pp_{k,n}$ (resp. $\qq_{k,n}$). 
In the Lie algebra vocabulary, $P_{k,n}$ and $Q_{k,n}$ (resp. $\pp_{k,n}$ and $\qq_{k,n}$)
are parabolic subgroups of $\GL_n$ (resp. parabolic subalgebras of
$\gl(V)$) and $\Lie (P_{k,n})=\pp_{k,n},\,
\Lie(Q_{k,n})=\qq_{k,n}$. In fact, all the content of this section can
easily be generalized to any parabolic subalgebra of $\gl(V)$ and a
corresponding nested Hilbert scheme. Namely, the parabolic subalgebra stabilizing a partial flag $F_0\subset F_{k_1}\subset \dots \subset F_{k_{\ell}}\subset F_n$ ($\dim F_{j}=j$) is in correspondence with the nested Hilbert scheme with length $n-k_{\ell}\leqslant \dots \leqslant n-k_{1}\leqslant n$. 

In Definition \ref{defFoncteurNCyc} and Proposition \ref{defFoncteurNCyc2}, we define a scheme $\MNTilde^{cyc}(\wfr)$, whose $\K$-points are the triples $(X,Y,v)$ with
$(X,Y)\in \MN(\wfr)$ and $v\in V$ 
 is a cyclic vector for the pair of endomorphisms $X,Y$. 

In Section \ref{Definitions_func}, we also describe an action of the group $P_{k,n}$ (resp. $Q_{k,n}$) on the scheme $\MNTilde^{cyc}(\pp_{k,n})$ (resp. $\MNTilde^{cyc}(\qq_{k,n})$).  Set-theoretically, this action is given by ${g\cdot} (X,Y,v)=(gXg^{-1},gYg^{-1}, gv)$.
The following theorem asserts that a $GIT$ quotient in the sense of
Mumford \cite{Mu} exists, and that
the quotients are nested punctual Hilbert schemes. 

\begin{thm}\label{thmcorres}
  \begin{enumerate}
  \item 
  The geometric quotients $\quo:\MNTilde^{cyc}(\pp_{k,n})\twoheadrightarrow
  \MNTilde^{cyc}(\pp_{k,n})/P_{k,n}$ and $\quoprime:\MNTilde^{cyc}(\qq_{k,n})\twoheadrightarrow
  \MNTilde^{cyc}(\qq_{k,n})/Q_{k,n}$ exist and they are principal
  bundles  locally
  trivial for the Zariski topology. 
\item There exist surjective morphisms 
\begin{displaymath}
\widetilde{\pi}_{k,n}:\MNTilde^{cyc}(\pp_{k,n})\twoheadrightarrow
      \sxno{n-k},
\end{displaymath}
\begin{displaymath}
\widetilde{\pi}'_{k,n}:\MNTilde^{cyc}(\qq_{k,n})\twoheadrightarrow \sxnoemb{n-k}.
\end{displaymath}
\item There exist isomorphisms $\isompHilb:\MNTilde^{cyc}(\pp_{k,n})/P_{k,n} \stackrel{\sim}{\to}
  \sxno{n-k}$ and  $\isomqHilb:\MNTilde^{cyc}(\qq_{k,n})/Q_{k,n} \stackrel{\sim}{\to}
  \sxnoemb{n-k}$. These isomorphisms identify the projections to 
the Hilbert schemes with the geometric quotients, \emph{i.e.}
  $\isompHilb \circ \quo=\widetilde{\pi}_{k,n}$ and  $\isomqHilb \circ \quoprime=\widetilde{\pi}'_{k,n}$.
\end{enumerate}
\end{thm}

\subsection{Functorial definitions}
\label{sec:descriptionsFonctorielles}
Hilbert schemes are often defined through their functor of points (see
\cite{EH} or \cite{St} for an introduction). We will use this setting
to prove Theorem \ref{thmcorres}. 
A useful example for us is the functor of points of the $\K$-vector
space $V$. This is the functor which associates 
\begin{itemize} 
\item to any $\K$-algebra $A$, the set $V(A):=V\otimes_{\K}A\cong A^n$.
\item to any morphism $A \rightarrow B$, the natural map 
 $V(A) \rightarrow V(B)=V(A)\otimes_A B$,
  $v\mapsto v\otimes 1$
\end{itemize}
In particular, the functor represented by $M_n$
(resp. $V_{k}$, $\pp_{k,n}$) associates to any $\K$-algebra $A$, the
set $M_n(A)$ of $n\times n$-matrices with coefficients in $A$
(resp. $V_k(A):=V_{k}\otimes A\subset V(A)$, $\pp_{k,n}(A):=\{X\in
M_n(A)\, |\, X(V_k(A))\subset V_k(A)\}$), see \cite[Example 2.1]{St}.
In the following, we will usually only make explicit the value of the functors on objects, their value on morphisms then being  standard. 
For more involved examples, the notion of relative representability turrns out to be useful.

\subsubsection{Relative representability}
We recall from  \cite{Gro2} the notion of relatively representable
morphism of functors, with some obvious adjustments to fit our
context. We will use this language to prove the representability of our
functors. 

Let $F,G$ be functors from the category of $\K$-algebras 
to sets. Suppose that $F$ is a subfunctor of $G$, ie. for every 
$\K$-algebra $A$, $F(A)$ is a subset of $G(A)$. The inclusion $F\subset G$ 
is relatively representable if, for every $\K$-algebra $A$ and every $g\in G(A)$ , there exists a 
subscheme $Z\subset Spec(A)$ satisfying the following property: 
for every $\phi:A\rightarrow B$, the morphism 
$Spec(B)\rightarrow Spec(A)$ factorizes through $Z$
if and only if the element $f\in G(B)$ defined by  $f=\phi_*(g)$
satisfies $f\in F(B)$. Grothendieck,  \cite[Lemme 3.6]{Gro2} proves that if $G$ is
representable and if $F\subset G$ is relatively representable, then 
$F$ is representable. 

In intuitive words, a relatively representable 
subfunctor $F\subset G$ is a subfunctor of $G$ defined by subscheme conditions on the
base. We illustrate this through the following elementary lemma.
\begin{lm}\label{def:Pkn}
The functor which maps a $\K$-algebra $A$ to the set $P_{k,n}(A):=\{X\in \pp_{k,n}(A)\, |\, \det X\textrm{ is invertible}\}$ is representable. The corresponding scheme is $P_{k,n}$.
\end{lm}
\begin{proof}
In the previous setting, we let $G(A):=\pp_{k,n}(A)$ and
$F(A):=P_{k,n}(A)$. Given $A$ and $g\in G(A)$, we set $Z:=\{p\in
Spec(A)| \det g\notin p \}$. Obviously, $Z$ is an open subscheme of
$Spec(A)$. For every $\phi:A\rightarrow B$, we consider the element
$f:=\phi_*(g)\in G(B)$. We have $f\in F(B)\Leftrightarrow \det
\phi_*(g)=\phi(\det g) \textrm{ is invertible }\Leftrightarrow \forall
p\in Spec(B)\, \det g\notin \phi^{-1}(p)$, that is, the comorphism
$Spec(B)\rightarrow Spec(A)$ factorizes through $Z$. In particular,
$F\subset G$ is relatively representable, hence $F$ is representable
by a subscheme of  $\pp_{k,n}$. The $\K$-points of
$F$ are those of the open subscheme $P_{k,n}\subset \pp_{k,n}$. Hence 
$P_{k,n}$  with the open subscheme structure represents $F$.
\end{proof}

This also applies when $F$ is a subfunctor of $G$ 
defined by the inclusion of two families 
according to the following lemma, proved in \cite[Lemma 1.1]{Kee}.

\begin{lm}\label{inclusionIsClosedCondition}
  Let $X\subset Spec(A)\times W$, $Y\subset Spec(A)\times W$ be two families of
  subschemes of a scheme $W$  with $X$ finite and flat over $Spec (A)$. There exists a
  subscheme $Z\subset Spec(A)$ such that, for every morphism
  $f:Spec(B)\rightarrow Spec(A)$, the following two conditions are
  equivalent:
  \begin{listecompacte}
    \item $f$ factorizes through $Z$
    \item $X \times_{Spec(A)}Spec(B) \subset  Y \times_{Spec(A)}Spec(B)$
  \end{listecompacte}
\end{lm}

\begin{prop}\label{foncteurHilbertEstRepresentable}
Let $n_1\geq n_2\dots \geq n_j > 0$ be integers. 
Let $F^{n_1,\dots,n_j}$ be the functor from
$\K$-algebras 
to sets defined by $F^{n_1,\dots,n_j}(A)=\{(I_1,\dots,I_j)\}$ where
\begin{listecompacte}
\item for every $i$, $I_i \subset A[x_1,\dots,x_d]$ is an  ideal,
\item $A[x_1,\dots,x_d]/I_i$ is locally free on $A$ of rank
  $n_i$,
\item $(x_1,\dots,x_d)^{n_i}\subset I_i$,
\item $I_1\subset I_2 \dots \subset I_j$.
\end{listecompacte}
Then $F^{n_1,\dots,n_j}$ is representable. 
\end{prop}
\begin{proof}
  For $j=1$, the functor $F^{n_1}$ parametrizes
  families of punctual subschemes of length $n_1$ in the closed
  subscheme $W$ defined by 
  the ideal $(x_1,\dots,x_d)^{n_1}$ in the affine space $Spec \ \K[x_1,\dots,x_d]$. 
It follows that this functor is
  representable by the Hilbert scheme $W^{[n_1]}=W_0^{[n_1]}$
  We then proceed by induction. Let $G^{n_1,\dots,n_j}$ be the functor defined
  similarly to $F^{n_1,\dots,n_j}$, except that we replace the condition $I_1\subset
  I_2 \dots \subset I_j$ with
  the condition  $I_1\subset I_2 \dots \subset I_{j-1}$. The functor
  $G^{n_1,\dots,n_j}$ is representable by $X_1\times X_2$,
 where $X_1$
  represents $F^{n_1,\dots,n_{j-1}}$, well defined by induction, and
  $X_2$ represents  $F^{n_j}$. The inclusion of functors
  $F^{n_1,\dots,n_j}\subset G^{n_1,\dots,n_j}$ is defined by the extra
  condition  $I_{j-1}\subset I_j$. According to the last lemma
  \ref{inclusionIsClosedCondition}, 
  this
  corresponds to a subscheme condition on the base of the families,
  ie. $F^{n_1,\dots,n_j}\subset G^{n_1,\dots,n_j}$ is relatively representable. It follows that  $F^{n_1,\dots,n_j}$ is representable.
\end{proof}

\subsubsection{Definitions}
\label{Definitions_func}
The functorial 
description of the Hilbert scheme $\sn$ is classical, but we need to 
precise the functorial description of $\MNTilde^{cyc}$ and 
of the variants $\sno, \skno$ of the Hilbert scheme that we use.

Consider the Hilbert-Chow morphism $\sn \rightarrow
Sym^n(\mathbb{A}^2)$, and compose it with the natural 
map $Sym^n(\mathbb{A}^2) \rightarrow Sym^n(\mathbb{A}^1)\times
Sym^n(\mathbb{A}^1)$.
We obtain a morphism 
$\rho:\sn \rightarrow Sym^n(\mathbb{A}^1)\times Sym^n(\mathbb{A}^1)$
which set-theoretically sends a subscheme $z_n$ to the tuples of
coordinates $(\{x_1,\dots,x_n\},\{y_1,\dots,y_n\})$  where $(x_i,y_i)$
are the points of $z_n$ counted with multiplicities.
A morphism $Spec\ R\rightarrow \sn$ factorizes through
$\rho^{-1}(0,0)$ if the corresponding ideal $I(Z) \subset R[x,y]$ 
satisfies $(x^n,y^n)\in I(Z)$. However, this property gives a special
status to the lines $x=0$ and $y=0$ as shown by the following
example, whose verification is straightforward. 

\begin{ex}
  Let $R=\K[a,b]/(ab,b^2)$ and $I=(y+ax+b,x^2)\subset R[x,y]$.
Then
  $ x^2\in I$, $y^2\in I$, but for any $t\in \K^*$, $(x+ty)^2\notin
  I$. 
\end{ex}

Consequently, we do not define $\sno$ as being $\rho^{-1}(0,0)$  
and we ask for a coordinate-free definition. The dimension of the 
ambient space $S$ plays no role in the definition. We shall give a 
general definition for the Hilbert scheme  $\zno$ parametrizing  
subschemes $z_n$ of length $n$ in a scheme $Z$ of any dimension $d$ 
supported on a smooth point $o\in Z$.

For this, we recall the well-known remark that a subscheme $z_n$ of
length $n$ in a scheme $Z$ is supported on a smooth point $o\in Z$ if
and only if $I(o)^n \subset I(z_n)$, ie. if $z_n$ is a subscheme of
$Spec\, \K[x_1,\dots,x_d]/(x_1,\dots,x_d)^n$ where $d$ is the dimension of $Z$
at $o$. This leads to the following definitions for the 
localized Hilbert scheme $\zno$ and the localized nested Hilbert
scheme $\zoemb$. Of course, they include our two main objects of study
$\skno$ and $\sknoemb$ with $Z=S=\mathbb A^2$ and $o=(0,0)$.

\begin{defi}
\label{defFoncteurHilbertLocal}
Let $Z$ be a scheme over $\K$, $o\in Z$ a smooth point such that the
local dimension of $Z$ at $o$ is $d$. 
 The Hilbert scheme $\zno$  is the scheme that represents the functor $F^{n}$ 
of Proposition \ref{foncteurHilbertEstRepresentable}. 
Let  $n_1\geq n_2\dots \geq n_j > 0$ be integers. 
The Hilbert
scheme $\zoemb$ is the scheme which represents the functor 
$F^{n_1,n_2,\dots,n_j}$.
\end{defi}

As long as we consider topological properties, a superscript $1$ plays no
role since the schemes  $\sxno{1}$ and $\sno$ are homeomorphic. In
fact, the following proposition shows they are even isomorphic as varieties.
\begin{prop}\label{s1no_var}Let  $(\sxno{1})_{red}$ and $(\sno)_{red}$ 
  be the varieties obtained from $\sxno{1}$ and $\sno$ with the reduced
  induced closed subscheme structure.  Then  $(\sxno{1})_{red}\cong(\sno)_{red}$ 
\end{prop}
\begin{proof}
  The functor $F^{n,1}$ associated with $\sxno{1}$ is defined by 
$F^{n,1}(A)=\{(I_1,I_2)\subset A[x,y]$ with $(x,y)^n\subset I_1\subset I_2$, $(x,y)\subset I_2$, $A[x,y]/I_1$ locally free of rank
$n$, $A[x,y]/I_2$ locally free of rank
$1\}$. In particular, $I_2=(x,y)$ is the only possibility. In other
words, if  $F^n$ denotes the functor
associated with $\sno$, then  $F^{n,1}$ can be seen as a subfunctor of $F^n$
defined by the condition $I_1\subset (x,y)$.
By Keel's lemma~\ref{inclusionIsClosedCondition}, this inclusion is
relatively representable and  $\sxno{1}$ is a closed subscheme of
$\sno$. When $A=\K$, $\K[x,y]/(x,y)^n$ is a local ring with maximal
ideal $(x,y)$. It follows that the inclusion $I_2 \subset (x,y)$ is always
satisfied or equivalently, that the embedding $\sxno{1}\subset \sno$ identifies the
$\K$-points on both sides. This proves the proposition.
\end{proof}

\begin{defi} \label{defFoncteurNCyc}
Let
  $A$ be a $\K$-algebra. Let $V(A)$, $V_k(A)$ and $\pp_{k,n}(A)$ be as in the beginning of Section \ref{sec:descriptionsFonctorielles}. Consider the functor $m$ from
  $\K$-algrebras to sets where $m(A)$ is 
\[\left\{\begin{array}{l}(X,Y,v)\in\\  \pp_{k,n}(A)\times \pp_{k,n}(A)\times V(A)\end{array}\!\!\left|\begin{array}{c}
[X,Y]=0,\; X^n=X^{n-1}Y=...=Y^n=0,\\ 
  (X^{(k)})^{n-k}=\cdots=(Y^{(k)})^{n-k}=0 \textrm{ on } V/V_k(A)\\  
 ev_n\textrm{ and }ev_{n-k} \textrm{ are surjective} \end{array}\right.\right\}\]
where \[ev_n:\func{A[x,y]}{V(A)\simeq A^n}{P(x,y)}{P(X,Y)(v)} \textrm{ and }  
ev_{n-k}:\func{A[x,y]}{V(A)/V_{k}(A)\simeq A^{n-k}}{P(x,y)}{P(X,Y)(v)+ V_k(A)}\] are the natural evaluation morphisms.
\end{defi}

\begin{prop}(Functorial definition of
  $\MNTilde^{cyc}(\pp_{k,n})\subset \pp_{k,n}\times \pp_{k,n}\times
  V$).\label{defFoncteurNCyc2}
  The functor $m$ is representable by a scheme  $\MNTilde^{cyc}(\pp_{k,n})$. 
\end{prop}
\begin{proof} 
We give a sketch of the proof. 
Let $m'$ be the functor given by the same conditions as $m$
except the surjectivity of $ev_n$ and $ev_{n-k}$.
In view of \cite[Example~2.1]{St}, $m'$
is representable by a
closed affine subscheme of $\pp_{k,n}\times\pp_{k,n}\times V$. 
Then, the inclusion $m\subset m'$ is defined by surjectivity
conditions, or equivalently by the invertibility of
some determinant. It follows that this inclusion of functors is
relatively representable, using the same argument as in the proof of Lemma~\ref{def:Pkn}. 
\end{proof}

The first point of the following lemma shows that the closed points of 
$\MNTilde^{cyc}(\pp_{k,n})$ are the expected triples
$(X,Y,v)$. Since, on $\K$-points, we require $X$ and $Y$ to be
nilpotent, it could seem natural in the above
definition of the functor $A \mapsto m(A)$ to replace the condition $X^n=X^{n-1}Y=...=Y^n=0$ with 
the simpler condition $X^n=Y^n=0$.
The second point of the lemma shows that this would 
add extra embedded
components to $\MNTilde^{cyc}(\pp_{k,n})$
and we are not interested in these components. 

\begin{lm}
(i) Let $X,Y\in M_n(\K)$ be a pair of nilpotent commuting matrices. Then $X^iY^{n-i}=0$ for all $i\in[\![0,n]\!]$.\\
(ii) The above conclusion may fail when replacing $\K$ by an arbitrary (even noetherian) $\K$-algebra $R$.
\end{lm}
\begin{proof}
(i) From reduction theory, it is an elementary fact that $X$ and $Y$
are simultaneously strictly upper trigonalisable. Hence the
equalities. 
\\
(ii) Take $R=\K[a,b]/(ab,b^2)$, $X=\left(
    \begin{array}{cc}
0 & 0\\
1& 0
    \end{array}
\right )$, $Y=\left(
  \begin{array}{cc}
b& 0\\
a& b
  \end{array}
\right )$. Then $X^2=Y^2=0$ and $XY=YX=\left (
  \begin{array}{cc}
0 & 0\\
b& 0
  \end{array}
\right )$.
\end{proof}

Finally, we can define the action of $P_{k,n}$ on $\MNTilde^{cyc}(\pp_{k,n})$, \emph{i.e.} the morphism $\gamma: P_{k,n}\times  \MNTilde^{cyc}(\pp_{k,n})\to  \MNTilde^{cyc}(\pp_{k,n})$ at 
  the functorial level. Let
  $g\in P_{k,n}(A)$,
$t=(X,Y,v)\in m(A)$ so $(g,t)\in Hom(Spec\ A,P_{k,n}\times
  \MNTilde^{cyc}(\pp_{k,n}))$. Then the element $t'=(X',Y',v')\in
  m(A)$, image of $(g,t)$ by the action morphism $\gamma$, is
  $X'=gXg^{-1},Y'=gYg^{-1},v'=gv$.

\subsection{The Hilbert scheme as a geometric quotient }
\label{sec:proof-theorem-ref}

In this section, we prove Theorem \ref{thmcorres}

  The cases of $\MNTilde^{cyc}(\pp_{k,n})$ and  $\MNTilde^{cyc}(\qq_{k,n})$ are
  similar and we consider only the first case. The strategy is
  the following. We first construct a categorical quotient. Using the 
  functorial properties of both the categorical quotient and the Hilbert
  scheme, we construct the
  isomorphism between $\MNTilde^{cyc}(\pp_{k,n})/P_{k,n}$ and
 $\sxno{n-k}$.
  Finally, using the description of the quotient via the Hilbert
  scheme, we show that the categorical quotient turns out to be a
  geometric quotient.



  Let $\Delta_{n-k}\subset \Delta_n$ be two sets of monomials
  $\{\delta_i
=x^{\alpha_i}
  y^{\beta_i}\}$ of respective cardinality $n-k$ and $n$. Let
  $\Delta=\{\Delta_{n-k},\Delta_n\}$. For each such $\Delta$, there is
an open subscheme $\MNTilde^{cyc}_\Delta\subset
\MNTilde^{cyc}(\pp_{k,n})$ whose support 
is the
locus where the evaluation morphisms $ev_{n-k}$ and $ev_n$ are surjective
using only the images of the monomials in $\Delta$. More precisely,
let $A[\Delta_i]$ be the free $A$-module with  basis $\Delta_i$.
The open subscheme $\MNTilde^{cyc}_\Delta$ corresponds to the subfunctor $m_\Delta(A)\subset m(A)$
containing the triples $(X,Y,v)\in m(A)$ such that  $ev_{\Delta_{n}}:A[\Delta_n]\to
A^n$, $\delta_i\mapsto (\delta_i(X,Y)(v))$ and  $ev_{\Delta_{n-k}}:A[\Delta_{n-k}]\to
A^{n-k}$, $\delta_i\mapsto (\delta_i(X,Y)(v))mod\ V_k(A)$ are surjective. 

  Recall that the surjectivity of the $A$-linear maps
  $ev_{\Delta_{n-k}}$ and $ev_{\Delta_{n}}$ is
  equivalent to their being an isomorphism (\cite{AtM}, Exercice 3.15),
  thus to their determinant being invertible in $A$. In particular,
  $\MNTilde_\Delta^{cyc}$ is defined by the nonvanishing of a determinant in 
  $\MN(\pp_{k,n}) \times \K^n$, hence it is  
  affine. 

  Since we have a covering of $\MNTilde^{cyc}(\pp_{k,n})$ 
  with open affine $P_{k,n}$-stable 
subschemes $\MNTilde_\Delta^{cyc}\simeq Spec\ B_\Delta$,
it is possible to
  construct a categorical quotient on each open subscheme as $\MNTilde_\Delta^{cyc}/P_{k,n}:=Spec \
  B_\Delta^{P_{k,n}}$ with the invariant functions. 
  Since the group is not
  reductive, $B_\Delta^{P_{k,n}}$ is not a priori finitely generated (and 
  we cannot apply \cite[Thm~1.1]{Mu}).
We have to show without the general theory 
  that the local quotients are algebraic (\emph{i.e.} of finite type over $\K$)
  and that the local constructions glue
  to produce a global categorical quotient.

Recall the functor $h$ which defines the Hilbert scheme $\skno$. 
If $\Delta$ is as above, there
is  a  subfunctor $h_{\Delta}$ of $h$. By definition, $h_\Delta(A)$
contains the pairs $(I,J)\in h(A)$ such that
$A[x,y]/I$ (resp. $A[x,y]/J$) is free on
$A$ of rank $n-k$ (resp. of rank $n$) and such that the monomials
$\delta_i$ in $\Delta_{n-k}$ (resp. in $\Delta_n$) form a basis of 
$A[x,y]/I$ (resp. $A[x,y]/J$). This is a relatively representable subfunctor, which is representable 
by an open subscheme $S_\Delta\subset S_0^{[n-k,n]}$.

There is a morphism of functors
$m\to h$ defined by \[(X,Y,v)\in m(A) \mapsto
(I=Ker(ev_{n-k}),J=Ker(ev_n))\in h(A)\] and a morphism of schemes 
$\widetilde{\pi}_{k,n}:\MNTilde^{cyc}(\pp_{k,n})\rightarrow \sxno{n-k}$ 
associated with the morphism of functors. By construction, this map 
is invariant under the action of $P_{k,n}$. From the universal
property of the categorical quotient, we obtain a factorisation
$\MNTilde^{cyc}_\Delta/P_{k,n}\rightarrow \sxno{n-k}$
whose image is in $S_\Delta$, hence the factorisation 
$\isompHilb_\Delta:\MNTilde^{cyc}_\Delta/P_{k,n}\rightarrow S_\Delta$.

To prove that $\isompHilb_\Delta$
is an isomorphism, we will construct an
inverse $\rho_\Delta$. 
Let $(I,J)\in h_\Delta(A)$. We choose a basis $b_1,\dots,b_n$ of $
A[x,y]/J$ such that  $b_{k+1},\dots,b_{n}$ is a basis of $A[x,y]/I$.
Such a basis exists since we can take $b_i$ to be 
the monomials in $\Delta$. 
If we replace each  element $b_i,i\leq k$ by a suitable 
combination $b_i+\sum_{j\geq k+1}a_{ij}b_j$, we may suppose that the
kernel $I/J$ of the map $A[x,y]/J \to A[x,y]/I$ is generated by 
$b_1,\dots,b_k$. 
This choice of our basis yields an 
effective isomorphism $A[x,y]/J\simeq A^{n}$. The
multiplication by $x$ and $y$ on $A[x,y]/J$ then correspond to matrices
$X,Y \in \pp_{k,n}(A)$. Choose $v=1\in A[x,y]/J$.
Then $(X,Y,v)\in m(A)$ 
and corresponds to a morphism $\nu:Spec\ A \to \MNTilde^{cyc}(\pp_{k,n})$. This morphism is not canonically defined because
of the arbitrary choice of the basis $b_1,\dots,b_n$. However, 
if $\nu_1$ and $\nu_2$ are two possible choices for the morphism
$\nu$, and if $\phi\in P_{k,n}(A)=Hom(Spec\;A,P_{k,n})$ is the decomposition 
matrix of the basis defining $\nu_1$ on the basis defining $\nu_2$, 
then $\nu_2=\gamma\circ(\phi,\nu_1)$, where $\gamma$
is the action morphism.
Since $\nu_1$ and $\nu_2$ differ by the action of $P_{k,n}(A)$, 
it follows that the morphism 
$\eta=\quo\circ \nu_1=\quo\circ \nu_2$ is well defined. 
The map which sends $(I,J)$ to $\eta$ is a morphism of functors.
This is the functorial 
description 
of a scheme morphism $\rho_\Delta:S_\Delta\rightarrow \MNTilde^{cyc}_\Delta/P_{k,n}$.
By construction, $\rho_\Delta$ and $\isompHilb_\Delta$ are mutually
inverse. 

Since we proved that our local quotients  $\MNTilde_\Delta^{cyc}/P_{k,n}$
are isomorphic to an open subscheme $S_{\Delta}$  
of the Hilbert scheme $S_0^{[n-k,k]}$, these local quotients
are algebraic. Gluing these local quotients
to form a global quotient $\MNTilde^{cyc}(\pp_{k,n})/P_{k,n}$ is
straightforward: this corresponds to the gluing of the open 
subschemes $S_{\Delta}$ in the Hilbert scheme $S_0^{[n-k,k]}$.




So far, we have proved that the Hilbert scheme  $S_0^{[n-k,k]}$ 
is a categorical quotient of  $\MNTilde^{cyc}$. 
There remains to prove that this quotient is locally trivial in the
Zariski topology. This will imply the remaining statements of the
theorem, namely that the quotient is geometrical
and the surjectivity of the quotient morphism. 
We shall prove the local triviality over $S_\Delta$. More precisely, 
we shall exhibit a pair of inverse isomorphisms $\phi_1,\phi_2$ to prove
that $S_\Delta \times  P_{k,n}$ and $\MNTilde_\Delta^{cyc}$ 
are isomorphic as schemes over $S_\Delta$.

Remark that we have constructed a (non-canonical) map $h_{\Delta}(A)\mapsto m(A)$
sending $(I,J)$ to $\nu$. Since this map depends functorially on $A$,
this functor corresponds to a section $s_\Delta:S_\Delta\rightarrow \MNTilde_\Delta^{cyc}$
of the map $\widetilde{\pi}_{k,n}:\MNTilde^{cyc}_\Delta \rightarrow
S_\Delta$. We define $\phi_1$ to be the composition 
\begin{displaymath}
  S_\Delta \times  P_{k,n} \stackrel{(s_\Delta,Id)}{\rightarrow } \MNTilde^{cyc}_\Delta  \times  P_{k,n} \rightarrow  \MNTilde^{cyc}_\Delta 
\end{displaymath}
where the second arrow is given by the group action.

The identity map $id_{\MNTilde_\Delta^{cyc}}$ on $\MNTilde_\Delta^{cyc}\simeq Spec(B_{\Delta})$, is an element of $m_{\Delta}(B_{\Delta})$.  
It yields an evaluation map $(ev_n)_1$ and
the following diagram, where $J$ is the kernel of $(ev_n)_1$
and $I$ is the kernel of $\psi\circ (ev_n)_1$. 
\[
\begin{array}{ccccccccccc}
&    &     &   & I          &        & I/J                  &   &        &       &      \\
&    &     &   & \hookdownarrow &        & \hookdownarrow           &   &        &       &      \\
   &  & J & \hookrightarrow        & B_\Delta[x,y] &
   \stackrel{(ev_n)_1}{\to} & V(B_\Delta) & 
   &  &  &  \\
   &  &  &      &  &
    & \phantom{\psi}\downarrow \psi & 
   &  &  &  \\
   &  &  &         &  &  & V(B_\Delta)/V_k(B_\Delta) &  &  &  &  
\end{array}.
\]
Using the map $s_\Delta \circ
\widetilde{\pi}_{k,n}:\MNTilde^{cyc}_\Delta
\rightarrow\MNTilde^{cyc}_\Delta $ instead of the identity map, 
we get a similar diagram with
$(ev_n)_2$ instead of $(ev_n)_1$ and $I,J$ unchanged. The 
 morphism 
$g=(ev_n)_1\circ((ev_n)_2)^{-1}\in \GL(V(B_\Delta))$ is then well
defined. Since $((ev_n)_2)^{-1}(Ker(\psi))=I$,  $g$ sends
$I/J=Ker(\psi)=V_k(B_{\Delta})$ to itself and $g\in P_{k,n}(B_{\Delta})=Hom(Spec(B_\Delta),P_{k,n})$. 
We define $\phi_2:  \MNTilde^{cyc}_\Delta  \rightarrow  S_\Delta
\times  P_{k,n} $ by $\phi_2=(\widetilde{\pi}_{k,n},g)$. 
By construction, the morphisms $\phi_1$ and $\phi_2$ are inverse. 


\subsection{From ${\MNTilde}$ to $\MN$} 

In the previous section, the Hilbert schemes $\skno$ ans $\sknoemb$ have been  
constructed as quotients of the schemes $\MNTilde^{cyc}(\pp_{k,n})$ and 
$\MNTilde^{cyc}(\qq_{k,n})$ which parametrize triples $(X,Y,v)$. In this
section, we show how to throw off the data $v$. From this point and until the end of the article, we only need to work with the underlying variety structure on our schemes. 
In particular, we will consider the following variety  for $\wfr=\pp_{k,n}$ or $\qq_{k,n}$:
\[\MN^{cyc}(\wfr):=\{(X,Y)\in \MN(\wfr)|\, \exists v\in V\textrm{ s.t. } (X,Y,v)\in \MNTilde^{cyc}(\wfr)\}.\]


\begin{lm}\label{freeaction}
\begin{itemize}
\item[(i)] The action of $P_{k,n}$ (resp. $Q_{k,n}$) on $\widetilde{\MN}^{cyc}(\pp_{k,n})$ (resp. $\widetilde{\MN}^{cyc}(\qq_{k,n})$) is free.
\item[(ii)] Let $v_1,v_2\in V$ such that $(X,Y,v_i)\in  \widetilde{\MN}^{cyc}(\pp_{k,n})$ (resp.  $\widetilde{\MN}^{cyc}(\qq_{k,n})$). Then $(X,Y,v_1)$ and $(X,Y,v_2)$ belong to the same $P_{k,n}$(resp. $Q_{k,n}$)-orbit.
\end{itemize}
\end{lm}
\begin{proof}
(i) Let $(X,Y,v)\in \widetilde{\MN}^{cyc}(\wfr)$ and $g\in \GL(V)$ stabilizing $(X,Y,v)$. Then $g$ stabilizes each $X^iY^j(v)$ and, since these elements generates $V$, we have $g=Id$.\\
(ii) 
Let $g:\func{V}{V}{P(X,Y).v_1}{P(X,Y).v_2}$. It is well defined since $\{P\in \K[x,y]|P(X,Y).v_i=0\}=\{P\in \K[x,y]|P(X,Y)=0\}$ by the cyclicity condition.
Moreover $g$ is linear and $g.v_1=v_2$.\\
For any $S\in\K[x,y]$, we have $gXg^{-1}(S(X,Y).v_2)=gXS(X,Y)(v_1)=g(S'(X,Y)(v_1))=S'(X,Y).v_2=X(S(X,Y)(v_2))$ where $S'=xS\in \K[x,y]$. In particular, $g$ stabilizes $X$ by cyclicity of $v_2$ and the same holds for $Y$.\\
A similar argument shows that any subspace $V_i\subset V$ stable under
$X$ and $Y$ is stabilized by $g$. The cyclicity property implies that
$g.v_1=S(X,Y)(v_1)$ and that $V_i$ is generated by $(R_l(X,Y)(v_1))_l$
for some polynomials $S,(R_l)_{l}$ of $\K[x,y]$. Then $g.V_i$ is
generated by $(g.R_l(X,Y)(v_1))_l=(R_l(X,Y)(g.v_1))_l=((R_l(X,Y)\times
S(X,Y))(v_1))_l=(S(X,Y)(R_l(X,Y)(v_1)))_l\subset V_i$. Hence $g$
stabilizes each such subspace $V_i$ and the result follows from the
definitions of $P_{k,n}$ and $Q_{k,n}$. 
\end{proof}

It follows from Lemma \ref{freeaction}(ii) that the following set-theoretical quotient map
\[\pi_{k,n}:\left\{\begin{array}{rcl}\MN^{cyc}(\pp_{k,n})&\rightarrow&\sxno{n-k}\\(X,Y)&\mapsto
    &(Ker(ev_{n-k}), Ker(ev_n)) \\&&(=\widetilde{\pi}_{k,n}(X,Y,v)\;
    \forall v\in V\textrm{ s.t. } (X,Y,v)\in
    \MNTilde^{cyc}(\pp_{k,n}))\end{array}\right.
\]
is well defined where $ev_{n-k}:\func{\K[x,y]}{\gl(V/V_k)}{P}{P(X^{(k)},Y^{(k)})}$ and $ev_{n}:\func{\K[x,y]}{\gl(V)}{P}{P(X,Y)}$.
This also allows to define $\pi'_{k,n}:\MN^{cyc}(\qq_{k,n})\rightarrow\sxnoemb{n-k}$.

\begin{prop}\label{corresirred}
$\pi_{k,n}$ induces a bijection between irreducible components of $\sxno{n-k}$ of dimension $m$ and irreducible components of $\mathcal \MN^{cyc}(\pp_{k,n})$ of dimension $m+(\dim \pp_{k,n}-n)$.
The same holds for $\pi'_{k,n}$, $\sxnoemb{n-k}$ and $\mathcal
\MN^{cyc}(\qq_{k,n})$.

\end{prop}
\begin{proof}
As usual, we give a proof only for $\pp_{k,n}$.

Let $Z_1, Z_2$ be varieties and $f:Z_1\rightarrow Z_2$ be an open surjective morphism with irreducible fibers. Then, the pre-image by $f$ of any irreducible component of $Z_2$ is irreducible (\emph{e.g.} see \cite[Proposition
1.1.7]{TY}). On the other hand, the image of any irreducible component of $Z_1$ by $f$ is irreducible. Hence $f$ induces a bijection between irreducible components of $Z_1$ and $Z_2$.

Then, since a geometric quotient by a connected group satisfies the above assumptions on $f$, we can apply the previous argument to $\widetilde{\pi}_{k,n}$. It also works for $pr:\func{\widetilde{\MN}^{cyc}(\pp_{k,n})}{\MN^{cyc}(\pp_{k,n})}{(X,Y,v)}{(X,Y)}$. The dimension statement follows since fibers of $\widetilde{\pi}_{k,n}$ are of dimension $\dim \pp$ (Lemma \ref{freeaction} (i)) and those of $pr$ are of dimension $n$ (given $(X,Y)$, the set $\{v|\,(X,Y,v)\in \MNTilde(\pp_{k,n})\}$ is open in $V$).
%
%
\end{proof}

The correspondence with commuting varieties allows us to see in an elementary way some
non-trivial facts on the Hilbert scheme. We give  an example.

\begin{prop}\label{schemasintermediaires}
  Given a pair $(z_{n-k},z_n)\in \sxno{n-k}$, there exists a chain of
  intermediate subschemes $z_{n-k}\subset z_{n-k+1}\subset
  \dots\subset z_n$. In other words, the projection map $\sxnoemb{n-k}\rightarrow \sxno{n-k}$ is surjective. 
The same holds for the projection map $\sxnoemb{n-k}\rightarrow \sxnoemb{n-k'}$ with $k\geqslant k'$.
\end{prop}
\begin{proof}
  The first assertion follows from the fact that any commuting pair 
 $(X_{|V_k},Y_{|V_k})\in \gl(V_k)$ is simultaneously trigonalizable by an element of $\GL_{V_k}\subset P_{k,n}$. Hence, in the new basis, it stabilizes the flag $V_1\subset V_2, \dots\subset V_k$.
The second one is the same argument applied to the pair $(X^{(k)},Y^{(k)})\in \gl(V/V_k)$.
\end{proof}

\begin{rem}\label{remSurLaSymetrie}  Note that there is a Lie algebra isomorphism between
  $\pp_{k,n}$ and $\pp_{n-k,n}$ (namely, minus the transposition with respect to the anti-diagonal). Hence the two varieties
  $\MN(\pp_{k,n})$ and $\MN(\pp_{n-k,n})$ are isomorphic. 

\[\xymatrix{ \MN^{cyc}(\pp_{n-k,n}) \ar@{^(->}[r]_{\quad\!\textrm{open}} \ar[d]_{\pi_{n-k,n}} &\MN(\pp_{k,n})& \MN^{cyc}(\pp_{k,n}) \ar@{_(->}[l]^{\textrm{open}\;}\ar[d]^{\pi_{k,n}} \\
\sxno{k} & &\sxno{n-k} }.\]

We use this duality in Lemma \ref{remonte1nmoins1} where we pull back
informations related to irreducibility from $\sxno{1}$ to
$\MN(\pp_{n-1,n})\cong\MN(\pp_{1,n})$. Eventually, this turns out to
be a key part of our proof of the irreducibility of $\sxno{n-1}$
(cf. Corollary \ref{coroIrredHilb1n}). 

However, the
  cyclicity condition breaks the symmetry
and there might be profound
  differences between $\MN^{cyc}(\pp_{k,n})$ and
  $\MN^{cyc}(\pp_{n-k,n})$, hence between $\sxno{n-k}$ and $\sxno{k}$.
  For instance,  $\sxyo{1}{3}$ and $\sxyo{2}{3}$ both contain a
  curvilinear locus as an open subvariety, and these curvilinear loci are
  isomorphic. On the boundary of this curvilinear locus, the two
  Hilbert schemes are quite different: when the scheme $z_3$ has
  equation $(x^2,xy,y^2)$ there is set theoretically only one length $1$
point
  $z_1$ in $z_3$, but there is a $\PP^1$ of $z_2$ with length $2$
  satisfying $z_2\subset z_3$. 

\end{rem}

\section{Technical lemmas on matrices}
\label{tech}

In this section, we collect technical results that will be used later
on. Most of these results aim to describe  $\mathfrak a^{nil}\subset \mathfrak a$,
where $\mathfrak a$ is a space of matrices commuting with a Jordan matrix
of type $\bolda\in \mathcal P(n)$ and $\mathfrak a^{nil}$ is the set of nilpotent matrices of $\mathfrak a$. In particular, we will  make frequent use of Lemmas  \ref{TA} and 
Proposition \ref{redred}. Parts of the results shown are well known in the more general framework of Lie algebras. Our goal is to translate this in the matrix setting and to provide a low-level understanding of the involved phenomena.

\begin{lm}\label{commons}
\item[(i)] $(M_n)^{nil}$ is an irreducible subvariety of codimension $n$ in $M_n$.
\item[(ii)] Assume that $\pp$ is the parabolic subalgebra defined by $\pp=\{X\in M_n| \,\forall j, \; X(V_{i_j})\subset V_{i_j}\}$ where the $i_j$ are $k+1$ indices satisfying $0=i_0\leqslant i_1\leqslant \dots i_k=n$. Then $X\in \pp$ is nilpotent if and only if the $k$ extracted matrices 
$$X_j=\left(\begin{array}{c c c}X_{i_{j-1}+1,i_{j-1}+1}&\cdots& X_{i_{j-1}+1,i_j} \\
\vdots & &\vdots\\
X_{i_j, i_{j-1}+1}& \cdots & X_{i_j,i_j}
\end{array}\right)\in M_{i_j-i_{j-1}}, \qquad 1\leqslant j\leqslant k, $$
are nilpotent.
\item [(iii)] If $\pp$ is a parabolic subalgebra of $M_n$ then $\pp^{nil}$ is an irreducible subvariety of $\pp$ of codimension $n$.
\end{lm}
\begin{proof}
(i) See \cite[Proposition~2.1]{Bas03} for an elementary proof of this classical fact.\\
(ii) First, note that $X_j$ can be viewed as the matrix of the endomorphism induced by $X$ on $V_{i_j}/V_{i_{j-1}}$. 
Then, as vector spaces, $$\pp\stackrel{v.s.}{\cong}  \mathfrak l \oplus \mathfrak n\qquad  \mbox{ where }\left\{ 
\begin{array}{l}\mathfrak l:=\prod_{j=1}^k(\End(V_{i_j}/V_{i_j-1}))\\
  \mathfrak n:=\{X\in \pp\,|\,X(V_{i_j})\subset
  V_{i_{j-1}}\}\end{array}\right. $$ and $\mathfrak n$ is a nilpotent
ideal of $\pp$.
Hence $X=X_{\mathfrak l}+X_{\mathfrak n}\in \pp$ is nilpotent if and only if $X_{\mathfrak l}$ is nilpotent. This is equivalent to the nilpotency of each $X_j$.

(iii) Up to base change, one can assume that $\pp$ satisfies the hypothesis of (ii). Thus $\pp^{nil}$ is isomorphic to $\prod_{j=1}^k(\End(V_{i_j}/V_{i_j-1}))^{nil}\times \mathfrak n$. It then follows from (i) that $\pp^{nil}$ is an irreducible subvariety of $\pp$ of codimension $\sum_{j=1}^k (i_j-i_{j-1})=n$.
\end{proof}
Let us explain (ii) in a more visual way.
\begin{ex}
A matrix of the form
$$X=\left(\begin{array}{c c c c c} a&b&c&d&e\\f&g&h& i& j\\0&0&k&l&m\\ 0&0&n&o&p\\0&0&q&r&s\end{array}\right) $$
is nilpotent if and only if the two following submatrices are nilpotent
$$X_1=\left(\begin{array}{c c} a &b\\ f&g\end{array}\right), \qquad X_2=\left(\begin{array}{c c c} k&l&m\\n&o&p\\ q&r&s\end{array}\right)$$
\end{ex}

Fix  an element $\bolda=(\lambda_1\geqslant \dots \geqslant \lambda_{d_{\bolda}})$  in $\mathcal P(n)$, the set of partitions of $n$. We define $X_{\bolda}\in M_n$ as the nilpotent element in Jordan canonical form associated to $\bolda$. In other words, in the basis $(f^i_j:=e_{\sum_{\ell=1}^{i-1} \lambda_{\ell}+j})_{\begin{subarray}{l}1\leqslant i\leqslant d_{\bolda}\\ 1\leqslant j\leqslant \lambda_i\end{subarray}}$, we have
\begin{equation}\label{Xbolda}  
X_{\bolda}(f^i_j)=\alter{f^i_{j-1}}{if $j\neq 1$,}0{else.}\end{equation}
For $Y\in M_n$, we denote the entries of $Y$ via
$Y.f^{i'}_{j'}=\sum_{(i,j)}Y^{i,i'}_{j,j'} f^i_j$ and use the following notation \[Y=\left(Y^{i,i'}_{j,j'}\right)_{(i,j), (i',j')}.\] An explicit characterization of $M_n^{X_{\bolda}}:=\{Y\in M_n| \;[X_{\bolda},Y]=0\}$ is given by the following classical lemma.

\begin{lm}\label{TA} $Y\in M_n^{X_{\bolda}}$ if and only if
  the following relations are satisfied:  
$$\left\{\begin{array}{l l}Y^{i,i'}_{j,j'}=0 &\mbox{ if $j>j'$ or $\lambda_i-j<\lambda_{i'}-j'$},\\
Y^{i,i'}_{j,j'}=Y^{i,i'}_{j-1,j'-1} &\mbox{ if $2\leqslant j\leqslant
  j'$ and $\lambda_i-j\geqslant \lambda_{i'}-j'$}. 
\end{array}\right.$$ 
Picturally, this means that $Y$ can be decomposed into blocks $Y^{i,i'}\in M_{\lambda_i,\lambda_{i'}}$ where
$$ Y^{i,i'}=\left(\begin{array}{c c c c}Y^{i,i'}_{1,1}& Y^{i,i'}_{1,2}& \dots & Y^{i,i'}_{1,\lambda_{i'}}\\
0 & Y^{i,i'}_{1,1}& \ddots & \vdots\\
\vdots & 0 & \ddots & Y^{i,i'}_{1,2}\\
\vdots & \vdots & \ddots & Y^{i,i'}_{1,1}\\
\vdots & \vdots & \vdots & 0\\
\vdots & \vdots & \vdots & \vdots\\
0&0&\dots&0
\end{array}\right)\mbox{ if $\lambda_i\geqslant \lambda_{i'}$},$$

$$ Y^{i,i'}=\left(\begin{array}{c c c c c c c}
0& \dots & 0 & Y^{i,i'}_{\lambda_{i},\lambda_{i'}}&\dots& Y^{i,i'}_{2,\lambda_{i'}} & Y^{i,i'}_{1,\lambda_{i'}}\\
0&\dots& \dots & 0 & \ddots & \ddots & Y^{i,i'}_{2,\lambda_{i'}}\\
\vdots & \dots & \dots & \dots & \ddots & \ddots &\vdots\\
0& \dots & \dots & \dots & \dots & 0 & Y^{i,i'}_{\lambda_{i},\lambda_{i'}}\\
\end{array}\right) \mbox{ if $\lambda_i\leqslant \lambda_{i'}$}.$$
\end{lm}

\begin{proof}
See \cite{TA} or \cite[Lemma~3.2]{Bas00} for a more recent account.
\end{proof}


Fix $\bolda\in \mathcal P(n)$.
For each length $\ell\in \N^{*}$ appearing in $\bolda$ (i.e. $\exists
i\in [\![1,d_{\bolda}]\!], \lambda_i=\ell$), we define
$\tau_{\ell}=\sharp\{i| \lambda_i={\ell}\}$.  Let
$W_{\ell}:=
\langle f^i_1| \lambda_i\geqslant \ell\rangle$. This is a filtration of $W:=W_1=\langle f^i_1|i\in [\![1,d_{\bolda}]\!]\rangle$ whose associated grading is given by the subspaces $W'_{\ell}:=\langle f^i_1| \lambda_i\geqslant \ell\rangle/\langle f^i_1| \lambda_i> \ell\rangle$ of dimension $\tau_{\ell}$.

It follows from Lemma \ref{TA} that each $W_{\ell}$ is stable under $M_n^{X_{\bolda}}$. 
Hence we have a Lie algebra morphism $M_{n}^{X_{\bolda}}\stackrel{pr_{ext}}\longrightarrow M_{d_{\bolda}}$
where the extracted matrix $pr_{ext}(Y)=Y_{ext}:=(Y_{1,1}^{i,i'})_{i,i'}$ can be seen as the element induced by $Y$ on $W=\Ker X_{\bolda}$.

\begin{lm}\label{extpara}
The image $(M_n^{X_{\bolda}})_{ext}$ of the morphism $pr_{ext}$ is the parabolic subalgebra 
\[\{Z\in M_{d_{\bolda}}\,|\,Z(W_{\ell})\subset W_{\ell}, \, \forall \ell\in \N^*\}.\]
\end{lm} 
\begin{proof}
It is an immediate consequence of Lemma \ref{TA}.
\end{proof}

Similarily we define the surjective (cf. Lemma \ref{TA}) maps
$M_{n}^{X_{\bolda}}\stackrel{pr_{\ell}}\rightarrow M_{\tau_{\ell}}=:
M_n^{X_{\bolda}} (\ell)\cong \gl(W_{\ell}')$ 
where 
\begin{equation}pr_{\ell}(Y)=Y(\ell):=(Y_{1,1}^{i,i'})_{((i,i')|\lambda_i=\lambda_{i'}=\ell)}\end{equation} can be seen as the element induced by $Y$ on $W'_{\ell}$.
We also define
$(M_n^{X_{\bolda}})\red:=\prod_{\ell} M_n^{X_{\bolda}}(\ell)$ and $pr\red$ as the surjective map: $\func{M_n^{X_{\bolda}}}{(M_n^{X_{\bolda}})\red}{Y}{Y\red=\prod_{\ell} Y(\ell)}$.
We have a natural section $\phi: (M_n^{X_{\bolda}})\red\rightarrow M_{n}^{X_{\bolda}}$ of the Lie algebra morphism $pr\red$ by setting 
$Z^{i,i'}_{j,j'}:=\alter{Y^{i,i'}}{ if $j=j'$, $\lambda_i=\lambda_{i'}$}{0}{ else}$  and $\phi((Y^{i,i'})_{i,i'}):=(Z^{i,i'}_{j,j'})_{(i,j),(i',j')}$. Hence, we can view $(M_n^{X_{\bolda}})\red$ as a subalgebra of $M_n^{X_{\bolda}}$ and 
\begin{equation}M_n^{X_{\bolda}}\stackrel{v.s.}{\cong}
  (M_n^{X_{\bolda}})\red\oplus \mathfrak n,\label{trois}
\end{equation}
where $\nf:=\Ker (pr\red)$.
A similar decomposition holds for $pr_{ext}$:  $M_n^{X_{\bolda}}\stackrel{v.s.}{\cong} (M_n^{X_{\bolda}})_{ext}\oplus 
\mathfrak n_1$ where $\nf_1:=\Ker (pr_{ext})$.

\begin{prop}\label{redred}
\item[(i)]$Y\in M_n^{X_{\bolda}}$ is nilpotent if and only if $Y\red$ is. In other words $(M_n^{X_{\bolda}})^{nil}\cong (M_n^{X_{\bolda}})\red^{nil}\times \mathfrak \nf_1$
\item[(ii)] $Y\in M_n^{X_{\bolda}}$ is nilpotent if and only if each $Y(\ell)\in M_{\tau_{\ell}}$ is.
\item[(iii)] $(M_n^{X_{\bolda}})^{nil}$ is an irreducible subvariety of $M_n^{X_{\bolda}}$ of codimension $d_{\bolda}$. 
\end{prop}
\begin{proof}
We associate to each basis element $f_{j}^i$ the weight
$w(f_{j}^i):=(j-\lambda_i,j)$. We order the weights
lexicographically. Lemma \ref{TA} asserts that $Y\in
M_n^{X_{\bolda}}$
is parabolic with respect to these weights, \emph{i.e.}
$Y(f_{b}^a)=\sum_{w(f_{b'}^{a'})\leq w(f_{b}^a)} c_{b'}^{a'}f_{b'}^{a'}$.
Remark that two elements $f_{j}^{i}$ and $f_{j'}^{i'}$ have the same weight if and only if
$\lambda_i=\lambda_{i'}$ and $j=j'$. We order the basis 
$f_{j}^{i}$ with respect to their weight. The base change from the
$f_{j}^{i}$ lexicographically ordered by their index $(i,j)$ to the
$f_{j}^{i}$ ordered by their weight transforms the matrix $Y$ into a
matrix $Z$.

Let $w$ be a weight and
$f_{j}^{i_1},f_{j}^{i_2},\dots f_{j}^{i_k}$ be the elements with weight
$w$ and $\ell:=\lambda_{i_m}$ (for any $m\in[\![1,k]\!]$). 
The diagonal block of $Z$ corresponding to the weight $w$ 
is $Y(\ell)$ (Lemma \ref{TA}). In other words, $Z$ and $Y\red$ have the same diagonal
blocks $Y(\ell)$, the difference being that the same block is repeated $\ell$
times in $Z$. In conclusion, $Y\red$ is nilpotent iff its diagonal
blocks $Y(\ell)$ are nilpotent, 
iff $Z$ and $Y$ are nilpotent. This proves $i)$ and $ii)$. 
Since $(M_n^{X_{\bolda}})\red\cong \prod_{\ell\in \N^*}M_{\tau_{\ell}}$ and $\sum_{\ell\in \N^*}\tau_{\ell}=d_{\bolda}$, Lemma \ref{commons} (i) allows us to conclude.
\end{proof}


%

In the Lie algebra vocabulary, $(M_n^{X_{\bolda}})\red$ is a reductive part (in $M_n$) of the centraliser of $X_{\bolda}$ in $M_n$ and $\mathfrak n$ is its nilpotent radical so \eqref{trois} can be written as a Lie algebra isomorphism $M_n^{X_{\bolda}}\stackrel{v.s.}{\cong} (M_n^{X_{\bolda}})\red\ltimes \mathfrak n$. See \cite{Pr} for an analogue of Proposition \ref{redred} (ii) valid for a general reductive Lie algebra.

\begin{ex}
Let $n=12$, $\bolda=(4,2,2,2,1,1)$ hence
$$X_{\bolda}={\footnotesize \left(\begin{array}{c@{\;\;}c@{\;\;}c@{\;\;}c@{\;}| c@{\;\;} c@{\;} |c@{\;\;} c@{\;}| c@{\;\;} c@{\;}| c@{\;}| c@{\;}} 
0&1&&&&&&&&&&\\ 
&0&1&&&&&&&&&\\ 
&&0&1&&&&&&&&\\
&&&0&&&&&&&&\\
\hline
&&&&0&1&&&&&&\\
&&&&&0&&&&&&\\
\hline
&&&&&&0&1&&&&\\
&&&&&&&0&&&&\\
\hline
&&&&&&&&0&1&&\\
&&&&&&&&&0&&\\
\hline
&&&&&&&&&&0&\\
\hline
&&&&&&&&&&&0\\
\end{array}\right)}, \quad M_n^{X_{\bolda}}\ni Y=
{\footnotesize \left(\begin{array}{c@{\;}c@{\,}c@{\,}c@{}| c@{\,}c@{} |c @{\,}c@{}| c@{\,}c@{}| c@{\,}| c} 
a&b&c&d&h_1&i_1&h_2&i_2&h_3&i_3&j_1&j_2\\ 
&a&b&c&&h_1&&h_2&&h_3&&\\
&&a&b&&&&&&&&\\
&&&a&&&&&&&&\\
\hline
&&h_4&i_4&e_1&f_1&k_1&l_1&k_2&l_2&m_1&m_2\\
&&&h_4&&e_1&&k_1&&k_2&&\\
\hline
&&h_5&i_5&k_3&l_3&e_2&f_2&k_4&l_4&m_3&m_4\\
&&&h_5&&k_3&&e_2&&k_4&&\\
\hline
&&h_6&i_6&k_5&l_5&k_6&l_6&e_3&f_3&m_5&m_6\\
&&&h_6&&k_5&&k_6&&e_3&&\\
\hline
&&&j_3&&m_7&&m_8&&m_9&g_1&n_1\\
\hline
&&&j_4&&m_{10}&&m_{11}&&m_{12}&n_2&g_2\\
\end{array}\right)}.$$
Here $d_{\bolda}=6$, $\tau_4=1$, $\tau_2=3$, $\tau_1=2$ and 

$$(M_n^{X_{\bolda}})_{ext}\ni Y_{ext}={\footnotesize \left(\begin{array}{c @{\;\;}|c@{\;\;} c@{\;\;} c@{\;}| c@{\;\;} c@{\;}}
a&h_1&h_2&h_3&j_1&j_2\\
\hline
& e_1 & k_1 & k_2& m_1&m_2\\
& k_3 & e_2 & k_4& m_3&m_4\\
& k_5 & k_6 & e_3& m_5&m_6\\
\hline
&  &  & & g_1&n_1\\
&  &  & & n_2&g_2\\
\end{array}\right)}, \quad (M_n^{X_{\bolda}})\red\ni Y\red={\footnotesize \left(\begin{array}{c @{\;\;}|c@{\;\;} c@{\;\;} c@{\;}| c@{\;\;} c@{\;}}
a&&&&&\\
\hline
& e_1 & k_1 & k_2& &\\
& k_3 & e_2 & k_4& &\\
& k_5 & k_6 & e_3& &\\
\hline
&  &  & & g_1&n_1\\
&  &  & & n_2&g_2\\
\end{array}\right)}.$$ 

$$\begin{array}{c} M_1\cong (M_n^{X_{\bolda}})(4)\ni Y(4)=\left(\begin{array}{c}a\end{array}\right),\\
M_3\cong (M_n^{X_{\bolda}})(2)\ni Y(2)=\left(\begin{array}{ccc} e_1 &k_1&k_2\\ k_3&e_2&k_4\\k_5&k_6&e_3\end{array}\right),\end{array} 
M_2\cong(M_n^{X_{\bolda}})(1)\ni Y(1)=\left(\begin{array}{c c} g_1 &n_1\\n_2& g_2\end{array}\right). $$
$Y$ is nilpotent if and only if $Y_{ext}$ is nilpotent if and only if the three matrices $Y(\ell)$ are nilpotent. 
This corresponds to the six($=d_{\bolda}$) independent conditions $$\left\{ Tr(Y(4))=a=0, \quad \begin{array}{l} Tr(Y(2))=0\\  e_1e_2+e_2e_3+e_3e_1-k_1k_3-k_4k_6-k_2k_5=0\\ \det(Y(2))=0\end{array},\quad \begin{array}{c} Tr(Y(1))=0\\ \det(Y(1))=0 \end{array}\right..$$
\end{ex}

\begin{defi} \label{defiwfrred}Let $\bolda\in \mathcal P(n)$ and $\wfr$ be a subspace of $M_n$ (\emph{e.g.} a Lie subalgebra of $M_n$ containing $X_{\bolda}$). We define the following vector spaces
$$\wfr^{X_{\bolda}}:=\wfr\cap M_n^{X_{\bolda}},\qquad\qquad 
(\wfr^{X_{\bolda}})\red:=\{Y\red| Y\in \wfr^{X_{\bolda}}\}.$$\end{defi} 

The following lemmas relate the geometry of $(\wfr^{X_{\bolda}})^{nil}$ to the one of $(\wfr^{X_{\bolda}})^{nil}\red$ or $(\wfr^{X_{\bolda}}(\ell))^{nil}$

\begin{lm}
\item[(i)] There exists a vector space $\mathfrak  n_2$ 
such that 
the following isomorphisms of algebraic varieties holds
$$ \wfr^{X_{\bolda}}\cong  (\wfr^{X_{\bolda}})\red\times \mathfrak n_2, \qquad
(\wfr^{X_{\bolda}})^{nil}\cong (\wfr^{X_{\bolda}})\red^{nil}\times \mathfrak n_2.$$

\item[(ii)] \label{codimwfrnil}
$(\wfr^{X_{\bolda}})^{nil}$ is irreducible if and only if $(\wfr^{X_{\bolda}})^{nil}\red$ is and 
$$\codim_{\wfr^{X_{\bolda}}}(\wfr^{X_{\bolda}})^{nil}=\codim_{(\wfr^{X_{\bolda}})\red}(\wfr^{X_{\bolda}})\red^{nil}$$
\end{lm}
\begin{proof}
(i) 
 Let  $\mathfrak n_2=\Ker ( (pr\red)_{|\wfr^{X_{\bolda}}})$.
The first equation follows and the statement about nilpotent elements is a consequence of Proposition \ref{redred} (i).\\
(ii) is a consequence of (i).
\end{proof}


Let $\wfr^{X_{\bolda}}(\ell):=pr_{\ell}(\wfr^{X_{\bolda}})=\{Y(\ell)\,|\, Y\in\wfr^{X_{\bolda}}\}\subseteq M_{\tau_{\ell}}$. We have a natural analogue of Proposition \ref{redred} (iii) in this case under some necessary restrictions.
\begin{lm}\label{wfrnil}
Let $\wfr$ be a subspace of $M_n$ such that the decomposition $\wfr^{X_{\bolda}}\red=\prod_{\ell} (\wfr^{X_{\bolda}})(\ell)$ holds.
\item[(i)]  
The variety $(\wfr^{X_{\bolda}})^{nil}$ is irreducible if and only if each $(\wfr^{X_{\bolda}}(\ell))^{nil}$ is and 
$$\codim_{\wfr^{X_{\bolda}}} (\wfr^{X_{\bolda}})^{nil}=\sum_{\ell} \codim_{\wfr^{X_{\bolda}}(\ell)}(\wfr^{X_{\bolda}}(\ell))^{nil}.$$
\item[(ii)] In particular if, for each $\ell$,
  $\wfr^{X_{\bolda}}(\ell)$ is isomorphic to $M_{\tau_{\ell}}$,
  $\pp_{k',\tau_{\ell}}$ or $\qq_{k',\tau_{\ell}}$ ($1\leqslant
  k' \leqslant \tau_{\ell}$)
then $(\wfr^{X_{\bolda}})^{nil}$ is irreducible and $\codim_{\wfr^{X_{\bolda}}} (\wfr^{X_{\bolda}})^{nil}=d_{\bolda}$.

\end{lm}

\begin{proof}
(i) follows from Lemma \ref{codimwfrnil}.\\
(ii) is then a consequence of Lemma \ref{commons}.
\end{proof}

\begin{rem}
\label{reml1l2}
The previous lemma is in general sufficient for our applications.
But, in some cases, we have $(\wfr^{X_{\bolda}})\red\subsetneq\prod_{\ell}\wfr^{X_{\bolda}}(\ell)$. A slightly less precise decomposition may remain valid in these cases. Define $\wfr^{X_{\bolda}}(\ell_1,\ell_2):=pr_{\ell_1,\ell_2}(\wfr^{X_{\bolda}})=\{(Y(\ell_1),Y(\ell_2))| Y\in\wfr^{X_{\bolda}}\}\subseteq \wfr^{X_{\bolda}}(\ell_1)\times \wfr^{X_{\bolda}}(\ell_2)$. 
Assume that there exists a decomposition of the form $\wfr^{X_{\bolda}}\red=(\wfr^{X_{\bolda}})(\ell_1, \ell_2)\times\prod_{\ell\notin\{\ell_1,\ell_2\}} (\wfr^{X_{\bolda}})(\ell)$.\\ 
Then $(\wfr^{X_{\bolda}})^{nil}$ is irreducible if and only if $(\wfr^{X_{\bolda}}(\ell_1, \ell_2))^{nil}$ and each $(\wfr^{X_{\bolda}}(\ell))^{nil}$ are. Then 
\begin{eqnarray}\label{decellgen}\codim_{\wfr^{X_{\bolda}}} (\wfr^{X_{\bolda}})^{nil}&=&\codim_{\wfr^{X_{\bolda}}(\ell_1, \ell_2)}(\wfr^{X_{\bolda}}(\ell_1, \ell_2))^{nil}\notag\\&&+\sum_{\ell\notin\{\ell_1,\ell_2\}} \codim_{\wfr^{X_{\bolda}}(\ell)}(\wfr^{X_{\bolda}}(\ell))^{nil}.\end{eqnarray}
\end{rem}

\section{Irreducibility of $\MN(\pp_{1,n})$ and $\sxno {n-1}$}
\label{secirrnmoins1}

The aim of this section is to prove that $\MN(\pp_{1,n})$ is
irreducible (Theorem \ref{irr1n}). We obtain 
as a corollary that a necessary and sufficient condition for 
the irreducibility of 
 $\MN(\pp_{k,n})$ and $\sxno{k}$ is 
$k\in \{0,1,n-1,n\}$ (Theorem \ref{thmCNSIrred}).

In this section, we will use the simplifying notation
$\pp:=\pp_{1,n}$. The strategy is the following. We introduce a 
variety $\mathcal M(\pp)$ of almost commuting matrices. 
Since $\mathcal M(\pp)$ is easily described as a graph, we get its irreducibility 
and its dimension. The dimensions of the components of $\MN(\pp)$ are controlled
through the equations defining $\MN(\pp)$ in $\mathcal
M(\pp)$. From this dimension estimate, we have a small list of candidates 
to be an irreducible component. We finally show that only one element
in this list defines an irreducible component.

In this section we assume $n\geqslant 2$.
%
%
Recall that $(e_1,\dots,e_n)$ is the canonical basis of $V=\K^n$,
$V_i=\langle e_1,\dots,e_i\rangle$. Also, we note $U_i:=\langle
e_{i+1},\dots,e_n\rangle$. 
We will mostly be interested in this section by $V_1= \K e_1$ and $U_1=\langle e_{2},\dots,e_n\rangle$. 
Recall also that $\pp=\pp_{1,n}=\{X\in \gl(V)\mid X(V_1)\subset V_1\}$.  
By virtue of Proposition \ref{corresirred}, we can study $\MN(\pp)$ in order to get informations on $\sxno{n-1}$. 

We have 
\begin{equation}\pp \stackrel{v.s.}{=}\gl(V_1)\oplus \Hom(U_1,V_1)\oplus \gl(U_1)\cong \K\oplus M_{1,n-1}\oplus M_{n-1}\label{decp0}\end{equation}

With respect to this decomposition, for any $X\in\pp$, we set $X=X_1+X_2+X_3$ where $X_1:=X_{\mid V_1}\in \gl(V_1)\cong \K$, $X_2\in \Hom(U_1,V_1)\cong M_{1,n-1}$ and $X_3\in \gl(U_1)\cong M_{n-1}$. That is
 
\begin{equation}\label{p} 
X=\decp{X_1}{X_2}{X_3}
\end{equation}
We will often identify $\Hom(U_1,V_1)$ with $E:=\langle {}^te_2,\dots,{}^te_n\rangle$.
Define 
$$\MCM(\pp):=\left\{(X,Y,j)\left| \begin{array}{c} (X,Y)\in \pp^2 , j\in \Hom(U_1,V_1) \\ X,Y \textrm{ nilpotent } \end{array},\; 
[X,Y]-\decp{0}{j}{(0)}
=0\right.
\right\}$$

The following Proposition and Corollary are prototypes for several
similar results of Section \ref{seclowbound}. The main ideas for this approach are taken from \cite{Zo}.
\begin{prop} \label{thmZo}
If $n\geqslant 2$, then $\MCM(\pp)$ is irreducible of dimension $n^2-2$ 
\end{prop}

\begin{proof}

Let us compute
\begin{equation}[X,Y]=\decp{0}{X_2Y_3-Y_2X_3}{[X_3,Y_3]}.\label{crochXY}\end{equation}
Hence, we have an alternative definition of $\MCM(\pp)$:
\begin{equation}\label{paraNg}(X,Y,j)\in \MCM(\pp)\Leftrightarrow \left\{\begin{array}{c}(X_3,Y_3)\in\MN(\gl(U_1)),\\ {X_1=Y_1=0}, \\ j=X_2Y_3-Y_2X_3.\end{array}\right.\end{equation}
In other words, $\MCM(\pp)$ is isomorphic to the graph of the morphism 
$$\begin{array}{rcl}\MN(M_{n-1})\times (M_{1,n-1})^2&\rightarrow& M_{1,n-1}\\ ((X_3,Y_3),(X_2,Y_2))&\mapsto& X_2Y_3-Y_2X_3\end{array}.$$
and the result follows from Theorem \ref{irrcomvar}.
\end{proof}

\begin{coro}\label{thmZobis}
The dimension of each irreducible component of $\MN(\pp)$ is greater or equal than $n^2-n-1$.
\end{coro}
\begin{proof}
If $n=1$, the result is obvious.

Else, we embed $\begin{array}{r c l}\MN(\pp)&\hookrightarrow &\MCM(\pp)\\ (X,Y)&\mapsto& (X,Y,0)\end{array}$. Hence, $\MN(\pp)$ is defined in $\MCM(\pp)$ by the $n-1$ equations $0=j\in M_{1,n-1}$ (cf. \eqref{paraNg}). Then, we conclude with Proposition \ref{thmZo}.
\end{proof}
%

Let us consider the set of \emph{$1$-marked partitions} of $n$ 
\[\mathcal P'(n):=\{(\lambda_1,(\lambda_2\geqslant \dots\geqslant \lambda_{d_{\bolda}}))\mid \sum_{i=1}^{d_{\bolda}}\lambda_i=n, \lambda_1\geqslant 1\}.\] Given $\bolda\in \mathcal P'(n)$, we let  $g^i_j:=e_{(\sum_{\ell=1}^{i-1} \lambda_{\ell})+j}$ for $\left\{\begin{array}{c}1\leqslant i\leqslant d_{\bolda}, \\ 1\leqslant j\leqslant \lambda_i\end{array}\right.$  and we define $X_{\bolda}\in \pp$ via
\begin{equation}\label{Xboldap}X_{\bolda}(g^i_j)=\left \{\begin{array}{l l}g^i_{j-1} &\mbox{ if $j>1$,}\\ 0 & \mbox{ if $j=1$.}\end{array}\right.\end{equation}
Note that these $X_{\bolda}$ with $\bolda\in \mathcal P'(n)$ are a priori different from the $X_{\bolda}$ with $\bolda\in \mathcal P(n)$ in spite of the similar notation used.

\begin{lm}[Classification Lemma]\label{classp1n}
Let $P:=\{x\in \pp\mid \det x\neq0\}$ be the connected subgroup of $G$ with Lie algebra $\pp$ and let $X$ be a nilpotent element of $\pp$.\\
There exists a unique $\bolda\in \mathcal P'(n)$ such that $P\cdot X=P\cdot X_{\bolda}$.
\end{lm}
\begin{proof}
Let us describe the $P$-action on $\pp^{nil}$.\\
Let $X=\decp{0}{X_2}{X_3}\in \pp^{nil}$ and $p=\decp{p_1}{p_2}{p_3}\in P$ (hence, $p_1\in \K^*$, $p_3\in \GL(U_1)\cong \GL_{n-1}$ and  $p^{-1}=\decp{p_1^{-1}}{-p_1^{-1}p_2p_3^{-1}}{p_3^{-1}}$). Then 
\begin{equation}\label{paction}p\cdot X=pXp^{-1}={ \decp{0}{p_2X_3p_3^{-1}+p_1X_2p_3^{-1}}{p_3 X_3 p_3^{-1}}}.\end{equation}
Hence, in order to classify $P$-orbits of $\pp^{nil}$, we can restrict ourselves to the case where $X_3$ is in Jordan normal form and study $P'\cdot X$ where $P'=\{p\in P\mid p_3\in\GL_{n-1}^{X_3}\}$. More precisely,
we fix $\boldmu\in \mathcal P(n-1)$ and $f^i_j:=e_{(\sum_{\ell=1}^{i-1} \mu_{\ell})+j+1}$ ($1\leqslant i\leqslant d_{\boldmu}, 1\leqslant j\leqslant \mu_i$)  and assume that
$$X_{3}(f^i_j)=\left \{\begin{array}{l l}f^i_{j-1} &\mbox{ if $j>1$,}\\ 0 & \mbox{ if $j=1$.}\end{array}\right.$$
Recall that we identify $\Hom(U_1,V_1)$ with $E=\langle{}^tf^i_j\rangle_{i,j}\cong \K^{n-1}$. The action of $\GL_{n-1}$ on this vector space that we consider is the natural right action. For any $p_3\in \GL_{n-1}^{X_{3}}$, we have $p_2X_3p_3^{-1}=p_2p_3^{-1}X_3$ and $\{p_2p_3^{-1}X_3\mid p_2\in E\}=\Img({}^tX_3)=\langle {}^tf_j^i\mid j\neq1\rangle$ for any $p_3\in \GL_{n-1}^{X_3}$.
On the other hand, set $i_0=\min\{i\mid {X_2(f_1^{i'})\neq0} \mbox{ for some $i'$ such that $\mu_i=\mu_{i'}$} \}$ (If $X_2=0$, set $i_0:=d_{\boldmu}+1$, $\mu_{i_0}=0$ and ${}^tf_1^{i_0}=0$). We have  
\begin{eqnarray}
\left\{p_1X_2p_3^{-1}\left| \begin{array}{c}p_1\in \K^*,\\ p_3\in \GL_{n-1}^{X_3}\end{array}\right.\right\}+\Img({}^tX_3)
&\stackrel{(p_1Id_{n-1}\subset \GL_{n-1}^{X_3})}{=}&\left\{X_2p_3^{-1}\mid p_3\in \GL_{n-1}^{X_3}\right\}+\Img({}^tX_3) 
\quad \notag\\
&\stackrel{\textrm{(Lemma~\ref{extpara})}}{=}&\left\langle{}^tf^i_1\,|\,\mu_i=\mu_{i_0}\right\rangle\setminus\{0\}\notag\\
&&+\left\langle {}^tf_j^i \mid j\neq 1 \mbox{ or } \mu_{i_0}>\mu_i \right\rangle\label{fiber}
\end{eqnarray}
As a consequence, the $P$-orbit of $X$ is uniquely determined by
$\boldmu$ and $i_0$. A representative of $P\cdot X$ is
$Y=\decp{0}{{}^tf_1^{i_0}}{X_3}$. Finally, an elementary base change in $P$
obtained by a re-ordering of the $(f^i_j)_{i,j}$ sends $Y$ on
$X_{\bolda}$ where $\bolda:=(\mu_{i_0}+1,(\mu_2\geqslant \cdots
\widehat{\mu_{i_0}}\cdots\geqslant \mu_{d_{\mu}}))$. 
\end{proof}
\begin{rem}\label{remfiber}
In the special case $X^0:=X_{\bolda^0}$ where
$\bolda^0:=(n,\emptyset)\in \mathcal P'(n)$, we also get $$\overline{P'\cdot X^0}=X^0_3+\Hom(U_1,V_1)$$
as a consequence of \eqref{fiber}, where $P'$ is the subgroup of $P$ defined in the previous proof.
\end{rem}


When $\bolda\in \mathcal P'(n)$, we say that $X_{\bolda}$ is in canonical form in $\pp$. Let 
\begin{equation}\MN_{\bolda}(\pp):=P\cdot(X_{\bolda}, (\pp^{X_{\bolda}})^{nil}).\label{MNboldapp}\end{equation}
Then
\begin{eqnarray}\dim \MN_{\bolda}(\pp) &=&\dim P\cdot X_{\bolda}+\dim \pxlnil\notag \\
&=&\dim \pp-\dim \pp^{X_{\bolda}}+\dim \pxlnil\notag\\
&=&\dim \pp-\codim_{\pp^{X_{\bolda}}}  \pxlnil. \label{codimnil}
\end{eqnarray}

\begin{lm}\label{irrpxbolda}
$$\MN(\pp)=\bigsqcup_{\bolda\in \mathcal P'(n)} \MN_{\bolda}(\pp)$$
Moreover, $(\pp^{X_{\bolda}})^{nil}$ and $\MN_{\bolda}(\pp)$ are irreducible and $\dim \MN_{\bolda}(\pp)=n^2-n+1-d_{\bolda}$.
\end{lm}
\begin{proof}
The decomposition into a disjoint union follows from Lemma \ref{classp1n}.

Let $\bolda\in \Pn'(n)$ and use notation of \eqref{Xboldap}.
In order to apply results of section \ref{tech}, we have to define a new basis $(f^i_j)$ in which $X:=X_{\bolda}$ is in canonical form for $M_n$ as in \eqref{Xbolda}. Set $i_0:=\max (\{i|\lambda_i>\lambda_1\}\cup\{1\})$ and 
$$f^i_j:=\left\{\begin{array}{c c}g^{1}_j & \mbox{ if } i=i_0 \\ g^{i+1}_j & \mbox{ if } i<i_0\\g^{i}_j & \mbox{ if } i>i_0 \end{array}\right.,\qquad \mu_i:=\left\{\begin{array}{c c}\lambda_1 & \mbox{ if } i=i_0 \\ \lambda_{i+1} & \mbox{ if } i<i_0\\\lambda_i & \mbox{ if } i>i_0 \end{array}\right..$$
In this basis, $X$ becomes $X_{\boldmu}$ with $\boldmu=(\mu_1\geqslant \dots\geqslant \mu_{d_{\bolda}})\in \mathcal P(n)$ and $\pp$ is defined in $M_n$ by the single property $Y\in \pp\Leftrightarrow Y(f_1^{i_0})\subset \K f_1^{i_0}$.
Hence, the subspace $(\pp^{X})\red$  (cf. Definition \ref{defiwfrred}) is also characterized in $(M_{n}^{X})\red$ by the single property
$$Y\red\in (\pp^{X})\red\Leftrightarrow Y\red(f_1^{i_0})\subset \K f_1^{i_0}.$$ 
In particular, letting $\tau_{\ell}:=\sharp\{i\,|\,\lambda_i=\ell\}=\sharp\{i\,|\,\mu_i=\ell\}$, we have $$\pp^{X}(\ell)\cong\alter{M_{\tau_{\ell}}}{if $\ell \neq \lambda_1$}
{\pp_{1,\tau_{\ell}}}{if $\ell=\lambda_1$}, \mbox{ and } (\pp^{X})\red=\prod_{\ell}\pp^{X}(\ell).$$
Then, Lemma \ref{wfrnil} (ii) provides the irreducibility statement for $(\pp^{X})^{nil}$ and hence for $\MN_{\bolda}(\pp)$. Together with \eqref{codimnil}, it also provides the dimension of $\MN_{\bolda}(\pp)$. 
%
\end{proof}

Combining this with corollary \ref{thmZobis}, we get that the irreducible components of $\MN(\pp)$ are of the form $\overline{\MN_{\bolda}(\pp)}$ where $\bolda\in \mathcal P'(n)$ has at most two parts ($d_{\bolda}\leqslant 2$). The unique irreducible component of maximal dimension is associated with $\bolda^0=(n,\emptyset)\in \mathcal P'(n)$.


There remains to show that \begin{equation}\label{MNbolda}\MN_{\bolda}(\pp)\subset \overline{\MN_{\bolda^0}(\pp)}\end{equation} when $\bolda$ has two parts. 
In order to prove this, we distinguish two cases.


\begin{lm}
If $\bolda=(\lambda_1,(\lambda_2))\in \Pn'(n)$ with $\lambda_1\leqslant\lambda_2+1$, property \eqref{MNbolda} is satisfied.
\end{lm}

\begin{proof}
For $(X_3,Y_3)\in \MN(\gl(U_1))$, we look at the fiber over $(X_3,Y_3)$ in $\MN(\pp)$ and $\overline{\MN_{\bolda^0}(\pp)}$: $$F_{X_3,Y_3}:=\{(X_2,Y_2)\in (\Hom(U_1, V_1))^2\mid (X_2+X_3,Y_2+Y_3)\in \MN(\pp)\},$$
$$F'_{X_3,Y_3}:=\{(X_2,Y_2)\in (\Hom(U_1, V_1))^2\mid (X_2+X_3,Y_2+Y_3)\in \overline{\MN_{\bolda^0}(\pp)}\}.$$

Since $F_{X_3,Y_3}=\{(X_2,Y_2)\mid {}^tX_3{}^tY_2={}^tY_3{}^tX_2\}$ (cf. \eqref{crochXY}) is a vector space, it is irreducible.
On the other hand, the two varieties $F_{X_3,Y_3}$ and $F'_{X_3,Y_3}$ are closed and satisfy $F'_{X_3,Y_3}\subset F_{X_3,Y_3}$. So 
\begin{equation}F_{X_3,Y_3}=F'_{X_3,Y_3}\Leftrightarrow \dim F_{X_3,Y_3}=\dim F'_{X_3,Y_3}.\label{eqdim}\end{equation}
We can compute the dimension of $F_{X_3,Y_3}$ in the following way: 
\begin{eqnarray*}
\dim F_{X_3,Y_3}&=&\dim (\Img ({}^tX_3)\cap \Img({}^tY_3))+\dim \Ker ({}^tX_3)+\dim \Ker({}^t Y_3)\\
&=&\dim \Img({}^tX_3)+\dim \Img({}^tY_3)-\dim (\Img({}^tX_3)+\Img({}^tY_3))\\&&+\dim \Ker ({}^tX_3)+\dim \Ker({}^t Y_3)\\
&=&2(n-1)-\dim (\Img({}^tX_3)+\Img({}^tY_3)).\end{eqnarray*}

Set $X^0:=X_{\bolda^0}$. Then, identifying $\Hom(U_1,V_1)$ with $\langle{}^te_2, \dots, {}^t e_n\rangle$ and using notation of \eqref{p}, we have $\Img({}^tX^0_3)=\langle {}^te_3,\dots {}^te_n\rangle$ and for any $Y_3\in (\gl(U_1)^{X^{0}_3})^{nil}$, the inclusion $\Img({}^tY_3)\subset \Img({}^tX^0_3)$ holds. Since  $\dim \Img({}^tX^0_3)=n-2$, we get $\dim F_{X^0_3,Y_3}=n$. An other consequence of the inclusion $\Img({}^tY_3)\subset \Img({}^tX^0_3)$ is the following: for any $X_2\in \Hom (U_1,V_1)$, there exists $Y_2\in\Hom (U_1,V_1)$ such that $(X_2,Y_2)\in F_{X_3^0,Y_3}$. Combining this with Remark \ref{remfiber}, we get that $X_3^0+X_2\in P.X_3^0$ for a general element $(X_2,Y_2)\in F_{X_3^0,Y_3}$ and 
\begin{equation*}\label{descrNl0}\overline{\MN_{\lambda^0}(\pp)}
=\overline{\GL(U_1)\cdot \left\{(X^0_3+ X_2, Y_3+Y_2)\left| \begin{array}{ c } Y_3\in(\gl(U_1)^{X_3^0})^{nil}\\(X_2,Y_2)\in F_{X_3^0,Y_3} \end{array}\right.\right\}.}\end{equation*}
In particular, a general element $(X,Y)$ of the irreducible variety $\overline{\MN_{\bolda^0}(\pp)}$ satisfies $\dim F_{X_3,Y_3}'=n$. 
Moreover, since $\MN(\gl(U_1))=\overline{\GL(U_1).(X_3^0,(\gl(U_1)^{X_3^0})^{nil})}$ (Theorem \ref{irrcomvar}), we see that any $(X_3,Y_3)\in \MN(\gl(U_1))$ lies in fact in $\overline{\MN_{\bolda^0}(\pp)}$ by considering the inclusion $\MN(\gl(U_1))\subset\MN(\pp)$ given by $X_2=Y_2=0$. Hence $F'_{X_3,Y_3}\neq\emptyset$ and 
\begin{equation}\forall (X_3,Y_3)\in \MN(\gl(U_1)), \quad\dim F'_{X_3,Y_3}\geqslant n.\label{leastfiber}\end{equation}

From now on, we fix $X:=X_{\bolda}$ and want to show that a general element $Y$ of $(\pp^X)^{nil}$ satisfies $(X,Y)\in \overline{\MN_{\bolda^0}(\pp)}$. This will prove the Lemma since $(\pp^X)^{nil}$ is irreducible (Lemma \ref{irrpxbolda}) and we will then have $(X,(\pp^X)^{nil})\subset \overline{\MN_{\bolda^0}(\pp)}$. Define $Z\in \pp$ by
$$Z(g^i_j)=\left\{\begin{array}{l l}g^2_{j-1}& \mbox{ if } i=1, \; j>1,\\ 
0 & \mbox { else.}  \end{array} \right.$$
We have $Z\in(\pp^X)^{nil}$  under the hypothesis made on $\bolda$ (Lemma \ref{TA}) and $\Img({}^tZ_3)+\Img({}^tX_3)=\langle g^1_2,\dots, g^1_{\lambda_1}, g^2_{2},\dots, g^2_{\lambda_2}\rangle$ so $\dim F_{X_3,Z_3}=n$. Since the application $\func{(\pp^X)^{nil}}{\N}{Y}{\dim F_{X_3,Y_3}}$ is upper semi-continuous, it follows from \eqref{leastfiber} that $W:=\left\{Y\in (\pp^{X})^{nil}\mid \dim F_{X_3,Y_3}=n=\dim F'_{X_3,Y_3}\right\}$ is a non-empty open subset of $(\pp^{X})^{nil}$. For $Y\in W$, we have $(X,Y)\in(X_3,Y_3)+F_{X_3,Y_3}\subset\overline{\MN_{\bolda^0}(\pp)}$ by \eqref{eqdim}.
%
\end{proof}


The following Lemma can be proved with purely matricial arguments. However, we find the given proof more interesting. It uses the isomorphism $\pp_{1,n}\cong\pp_{n-1,n}$ and enlighten a bit the correspondence between $\sxno{1}$ and $\sxno{n-1}$ mentioned in remark \ref{remSurLaSymetrie}.

\begin{lm}\label{remonte1nmoins1}
If $\bolda=(\lambda_1,(\lambda_2))\in \Pn'(n)$ with $\lambda_1\geqslant \lambda_2$, then Property \eqref{MNbolda} is satisfied.
\end{lm}
\begin{proof}
Seen as varieties, we have $\sxno{1}\stackrel{var}{\cong}\sno$ (Proposition \ref{s1no_var}).
In particular, the irreducibility of $\sxno{1}$ follows from that of $\sno$ \cite{Br,Pr} and $\MN^{cyc}(\pp_{n-1,n})$ is then also irreducible (Proposition \ref{corresirred}). 


We have a Lie algebra isomorphism given
by \begin{equation}\label{transiso}\psi':
  \func{\pp_{1,n}}{\pp_{n-1,n}}{X}{-s({}^tX)s^{-1}}\end{equation}
where $s$ is defined on the basis $(e_i)_{i\in [\![1,n]\!]}$ by
$s(e_i):=e_{n-i}$. In particular, the restriction
$\psi:\MN(\pp_{1,n})\rightarrow\MN(\pp_{n-1,n})$ is an isomorphism of
varieties. Moreover, $\psi(X,Y)$ has a cyclic vector if and only if
$({}^tX, {}^tY)$ does.

Note that
$\psi(\MN_{\bolda^0}(\pp_{1,n}))=\MN_{\bolda^0}(\pp_{n-1,n})$ and
that $\MN_{\bolda^0}(\pp_{1,n})$ is open in $\MN(\pp_{1,n})$.
It is then straightforward to check that $\psi(\MN_{\bolda^0}(\pp_{1,n}))\subset \MN^{cyc}(\pp_{n-1,n})$. Hence, it follows from Lemma \ref{irrpxbolda} and the irreducibility of the open subvariety $\MN^{cyc}(\pp_{n-1,n})\subset \MN(\pp_{n-1,n})$ that $\overline{\psi(\MN_{\bolda^0}(\pp_{1,n}))}=\overline{\MN^{cyc}(\pp_{n-1,n})}$.

Consider now $X_{\bolda}$ given by \eqref{Xboldap}. We can define $Y\in (\pp_{1,n})^{nil}$ via
\[Y(g^i_j):=\alter{g^{1}_j}{\textrm{if } i=2,}{0}{\textrm{if }i=1.}\]
Under our hypothesis on $\bolda$, we have $Y\in\pp_{1,n}^{X_{\bolda}}$ (Lemma \ref{TA}) and ${}^tg_1^1$ is a cyclic vector for $({}^tX_{\bolda}, {}^tY)$. In particular, $\psi(\MN_{\bolda}(\pp_{1,n}))\cap \MN^{cyc}(\pp_{n-1,n})\neq \emptyset$ so the irreducible subset $\psi(\MN_{\bolda}(\pp_{1,n}))$ is contained in $\overline{\MN^{cyc}(\pp_{n-1,n})}=\overline{\psi(\MN_{\bolda^0}(\pp_{1,n}))}$.
Since $\psi$ is an isomorphism, \eqref{MNbolda} is proved in our case.
\end{proof}
Finally, it follows from discussion above \eqref{MNbolda} that the following theorem holds.
\begin{thm} \label{irr1n}
The variety $\MN(\pp_{1,n})$ is irreducible of
dimension $ n^2-n=\dim \pp_{1,n}-1.$
\end{thm}

Hence, by Proposition \ref{corresirred}:
\begin{coro} \label{coroIrredHilb1n}
$\sxno{n-1}$ is an irreducible variety of dimension $n-1$.
\end{coro}


\begin{rem}
  The above corollary was already proved in \cite{CE} with other
  techniques (Bialynicki-Birula stratifications and Gr\"obner basis
  computations).
\end{rem}

\begin{thm} \label{thmCNSIrred}
  $\skno$ is irreducible if and only if $k\in \{0,1,n-1,n\}$. 
  $\MN(\pp_{k,n})$ is irreducible if and only if $k\in \{0,1,n-1,n\}$.
\end{thm}
\begin{proof}
  Since $\sno$ is irreducible \cite{Br,Pr}, since $\sxno{1}$ is
  homeomorphic to $\sno$ and $\sxno{n}$ is isomorphic to $\sno$, this
  proves together with Proposition \ref{irredOnlyFor1nMoins1} the
  assertion of the Theorem for $\skno$. The variety 
  $\MN(\pp_{k,n})$ is irreducible for $k=1$ by Theorem
  \ref{irr1n}. The transposition isomorphism of \eqref{transiso} implies that
  $\MN(\pp_{n-1,n})\simeq \MN(\pp_{1,n})$ is irreducible
  too. The varieties $\MN(\pp_{0,n})=\MN(\pp_{n,n})=\MN(M_n)$ are
  irreducible by Theorem \ref{irrcomvar}. Finally, the number of components in $\MN(\pp_{k,n})$ is greater or equal than the number of components of $\MN^{cyc}(\pp_{k,n})$ which is, in turn,
 equal to the number of components in $\sxno{n-k}$ (Proposition \ref{corresirred}). It follows that
  $\MN(\pp_{k,n})$ is reducible for $k\in \{2,\dots,n-2\}$
\end{proof}

\begin{coro}\label{thmCNSIrredq}
  $\sknoemb$ is irreducible if and only if $k\in \{n-1,n\}$ or $n\leqslant 3$. 
  $\MN(\qq_{k,n})$ is irreducible if and only if $k\in \{0,1\}$ or $n\leqslant 3$.
\end{coro}
\begin{proof}
  Note that $\sknoemb\simeq \skno$ for $k=n-1,n$ and $\MN(\qq_{k,n})=
  \MN(\pp_{k,n})$ for $k=0,1$. The ``if'' part then follows 
  from Theorem \ref{thmCNSIrred} and easy computations when
  $n\leqslant 3$.
For $k\geqslant 2$, we have a sequence of surjective morphisms
$\MN^{cyc}(\qq_{k,n})\rightarrow
\sxnoemb{n-k}\rightarrow\sxnoemb{n-2}\rightarrow \sxno{n-2}$ (Propositions
\ref{corresirred} and \ref{schemasintermediaires}). Since
$\sxno{n-2}$ is reducible when $n\geqslant 4$,
the corollary follows.
%
\end{proof}

\section{General lower bounds for the dimension of the components}
\label{seclowbound}

The goal of this section is to give lower bounds for the dimension of
the components of $\MN(\pp_{k,n})$, $\MN(\qq_{k,n})$,
$\sxno{k}$,$\sxnoemb{k}$ which are valid for all $k,n$.

Let $n\geqslant 2$ and $1\leqslant k\leqslant n-1$. 

\begin{prop}\label{dimqq}
Each irreducible component of $\MN(\qq_{k,n})$ has dimension at least $\dim\qq_{k,n}-1$. 
\end{prop}
\begin{proof}
We proceed by induction on $k$, the case $k=1$ being proved in Theorem \ref{irr1n} since $\qq_{1,n}=\pp_{1,n}$. 
Assume now $k\geqslant 2$. The proof mainly follows those of Proposition \ref{thmZo} and Corollary \ref{thmZobis}.

For any $X\in\qq_{k,n}$, we decompose $X=X_1+X_2+X_3$ as in \eqref{p}, with $X_3\in (\gl(U_1)\cap \qq_{k,n})\cong \qq_{k-1,n-1}$.
We let
\begin{eqnarray*}\MCM(\qq_{k,n})&:=&\left\{(X,Y,j)\in (\qq_{k,n}^{nil})^2\times \Hom(U_1,V_1) \,\left|\, [X,Y]-\decp{0}{j}{0}=0\right.\right\}
\end{eqnarray*} 
Then
$$(X,Y,j)\in \MCM(\qq_{k,n})\Leftrightarrow \left\{\begin{array}{c}(X_3,Y_3)\in\MN(\qq_{k-1,n-1}),\\ {X_1=Y_1=0}, \\ j=X_2Y_3-Y_2X_3\end{array}\right.$$

Thus $\MCM(\qq_{k,n})$ is isomorphic to the graph of 
$$\phi:\begin{array}{rcl}\MN(\qq_{k-1,n-1})\times(M_{1,n-1})^2&\rightarrow&M_{1,n-1}\\ ((X_3,Y_3),(X_2,Y_2))&\mapsto& X_2Y_3-Y_2X_3\end{array}.$$
Hence, by induction, each irreducible component of $\MCM(\qq_{k,n})$ has a dimension greater or equal than $\dim \qq_{k-1,n-1}-1+2(n-1)$.

Since $k-1\geqslant 1$, and $X_3,Y_3\in \gl(U_1)$ are nilpotent, we get $X_3(e_2)=Y_3(e_2)=0$. So the image of $\phi$ lies in $\langle {}^te_3, \dots{}^te_n\rangle$ and $\MN(\qq_{k,n})$ is defined in $\MCM(\qq_{k,n})$ by $n-2$ equations. Then the dimension of each of its irreducible component is greater or equal than $ \dim \qq_{k-1,n-1}-1+2(n-1)-(n-2)=\dim \qq_{k-1,n-1}+n-1=\dim \qq_{k,n}-1$.
\end{proof}
Hence, by Proposition \ref{corresirred}:
\begin{coro}\label{corodimsxno}
Each  irreducible component of $\sxnoemb{n-k}$ has dimension at least $n-1$ which is the dimension of the curvilinear component.
\end{coro}
\begin{rem}When $k=n$, Proposition~\ref{dimqq} provides a lower bound for the dimension of the nilpotent commuting variety of the Borel subalgebra $\qq_{n,n}$. In this case, a simpler proof is given by considering the bracket map $\nf\times\nf\rightarrow[\nf,\nf]$ where $\nf$ is the nilradical of a Borel $\bb$. Its fibers, in particular its null one which is equal to $\MN(\bb)$, are of dimension greater or equal than $2\dim \nf-\dim [\nf,\nf]=\dim \bb$  in an arbitrary semisimple Lie algebra.  When $\bb$ acts on $\nf$ with finitely many orbits, a computation similar to \eqref{codimnil} then shows that $\MN(\bb)$ is an equidimensional variety. This simplifies some of the arguments of \cite{GR}, where this result was first proved, since it allows to avoid Strategy 2.10 (2-3) in this case.
\end{rem}

Unfortunately, concerning $\pp_{k,n}$ we are only able to give the following less effective bound.
\begin{prop}
Each irreducible component of $\MN(\pp_{k,n})$ has dimension at least $\dim\pp_{k,n}-2$. 
\end{prop}
\begin{proof}
Let $$\MCM(\pp_{k,n}):=\left\{(X,Y,B)\in \pp_{k,n}^2\times
  \Hom(U_k,V_k) \left| [X,Y]-\decp{0}{B}{0}=0\right.\right\}.$$ 
Once again, we proceed in a similar way to Proposition \ref{thmZo}.

Hence, $\MCM(\pp_{k,n})$ is isomorphic to the graph of a morphism with an irreducible
domain of dimension $(k^2-1)+((n-k)^2-1)+2k(n-k)$ and $\MN(\pp_{k,n})$
is defined in $\MCM(\pp_{k,n})$ by $k(n-k)$ equations. Hence, the
dimension of each irreducible components of $\MN(\pp_{k,n})$ is
greater or equal than $k^2+(n-k)^2+k(n-k)-2=\dim \pp_{k,n}-2$.
\end{proof}

Finally, we have the following consequence concerning nested Hilbert schemes (cf. Proposition \ref{corresirred}).
\begin{coro}
Each irreducible component of $\sxno{n-k}$ has dimension at least $n-2$, which is the dimension of the curvilinear component minus one.
\end{coro}

Applying naively the same argument to a general parabolic subalgebra $\pp$ whose Levi part has $\ell$ blocks, one can show that the dimension of any irreducible component of $\MN(\pp)$ (resp. of the corresponding Hilbert scheme) has dimension at least $D-(\ell-1)$ with $D=\dim \pp-1$ (resp. $D=n-1$). We think that the correct dimension should be $D$ but were only able to prove this in special cases such as $\qq_{k,n}$. In fact, in this case as in some others, the extra codimension yielded by the Levi-blocks of size $1$ can be discarded easily, hence the optimal result. 

\section{Detailed study of $\sdeuxno$}

In the special cases $k=2$ and $k=n-2$, we have a more precise
estimate for the dimension of the components. The goal of this 
section is to describe the number and the dimension of the components 
for $\MN(\pp_{2,n})\simeq \MN(\pp_{n-2,n})$, $\MN(\qq_{2,n})$, 
$\sdeuxno$,$\sxno{n-2}$, $\sxnoemb{n-2}$. The general strategy is the same as in Section \ref{secirrnmoins1}.

Our first aim is a classification of orbits.

We identify $\qq_{2,n}$ with 
$$\gl(V_1)\oplus\Hom(U_1,V_1)\oplus\{X\in \gl(U_1)|X(e_2)\in \CC e_2\}\stackrel{v.s.}{\cong} \K\oplus M_{1,n-1}\oplus\qq_{1,n-1}.$$ Again, we decompose each $X\in \qq_{2,n}$ with respect to this direct sum \begin{equation}\label{decXp2}X=\decp{X_1}{X_2}{X_3}.\end{equation}
For any $\bolda=(\lambda_1, (\lambda_2\geqslant \dots\geqslant
\lambda_{d_{\bolda}}))\in \Pn'(n)$, we set
$\lambda_{d_{\bolda}+1}=0$. Let 
\begin{equation}\label{defpseconde}
\mathcal
P''(n):=\left\{(\bolda,l,\epsilon)\in \mathcal
  P'(n-1)\times\N\times\{0,1\}\left| 
\begin{array}{l}l=\lambda_i\mbox{ for some $i\in[\![2,
    d_{\bolda}+1]\!]$ }\\
\epsilon=1\Rightarrow (l>\lambda_1 \mbox{ or } l=0)
\end{array}\right.\right\}.\end{equation}
For $\boldmu=(\bolda,l,\epsilon)\in \mathcal P''(n)$, we define $g^i_j:=e_{\sum_{\ell=1}^{i-1}\lambda_{\ell}+j+1}$ and $i_{\boldmu}:=\min\{i'>1\mid l=\lambda_{i'}\}\in [\![2, d_{\bolda}+1]\!]$.
In the basis 
$(e_1,(g^i_j)_{\left(\begin{subarray}{l}1\leqslant i\leqslant d_{\bolda}\\ 1\leqslant j\leqslant \lambda_i\end{subarray}\right)})$, we define $X_{\boldmu}\in\qq_{2,n}$ via
$$ X_{\boldmu}(e_1):=0, \qquad
X_{\boldmu}(g^i_j):=\left\{\begin{array}{ll} g^i_{j-1} &\mbox{ if $j>1$}\\
  \epsilon\, e_1 &\mbox{ if $i=1$, $j=1$}\\  e_1 &\mbox{ if
    $i=i_{\boldmu}$ and $j=1$} \\ 0 & \mbox{
    else}\end{array}\right..$$
\[X_{\boldmu}={\scriptsize\left(\begin{array}{c!{\vrule width 1pt}cccc|cc|cccc|cc|}0 &\epsilon &0&\phantom{0}&\multicolumn{1}{c}{\cdots}&\phantom{0}&\multicolumn{2}{c}{0 \hspace{10pt} 1}&0&\phantom{0}&\multicolumn{1}{c}{\cdots}&\multicolumn{1}{c}{\phantom{0}}&\multicolumn{1}{c}{0}\\
\Hline{1pt}
& 0&1&&\\
&&\ddots&\ddots&\\
&&&\ddots&1\\
&&&&0\\
\cline{2-7}
&&&&&&&\\
&&&&&&&\\
\cline{6-11}
&&&&\multicolumn{1}{c}{\phantom{0}}&&& 0&1&&\\
&&&&\multicolumn{1}{c}{\phantom{0}}&&&&\ddots&\ddots&\\
&&&&\multicolumn{1}{c}{\phantom{0}}&&&&&\ddots&1\\
&&&&\multicolumn{1}{c}{\phantom{0}}&&&&&&0\\
\cline{8-13}
&&&&\multicolumn{1}{c}{\phantom{0}}&&\multicolumn{1}{c}{\phantom{0}}&&&&&&\\
&&&&\multicolumn{1}{c}{\phantom{0}}&&\multicolumn{1}{c}{\phantom{0}}&&&&&&\\
\cline{12-13}
&\undermat{\lambda_1}{\phantom{0}&\phantom{\cdots}&\phantom{\cdots}&\multicolumn{1}{c}{\phantom{0}}}
&\phantom{0} &\multicolumn{1}{c}{\phantom{0}}& \undermat{\lambda_{i_{\boldmu}}}{\phantom{0}&\phantom{\cdots}&\phantom{\cdots}&\multicolumn{1}{c}{\phantom{0}}}&
\end{array}\right)}\]
\medskip

Note that in the basis $(g^i_j)_{i,j}$ of $U_1$, we have $(X_{\boldmu})_3=X_{\bolda}$ in the notation of  \eqref{Xboldap}. We claim that $(X_{\boldmu})_{\boldmu\in \mathcal P''(n)}$ is a set of representatives of nilpotent orbits of $\qq_{2,n}$.

\begin{lm}\label{Xmu}
Each nilpotent element of $\qq_{2,n}$ (resp. $\pp_{2,n}$) is $Q_{2,n}$(resp. $P_{2,n}$)-conjugated to $X_{\boldmu}$ for some $\boldmu\in \mathcal P''(n)$.\\ Moreover $Q_{2,n}\cdot X_{\boldmu}=Q_{2,n}\cdot X_{\boldmu'}$ if and only if $\boldmu=\boldmu'$.
\end{lm}
\begin{proof}
Thanks to the inclusion $(\GL(V_2)\times Id_{U_2})\subset P_{2,n}$, we can trigonalize the $\gl(V_2)$-part of any element of $\pp_{2,n}$, hence each element of $\pp_{2,n}$ is $P_{2,n}$-conjugated to an element of $\qq_{2,n}$. Since $Q_{2,n}\subset P_{2,n}$, it is therefore sufficient to prove the result for $\qq_{2,n}$.

Let $X=X_1+X_2+X_3\in\qq_{2,n}$ be a nilpotent element. We have
$X_1=0$. The element $X_3$ is nilpotent so, up to conjugacy by an
element of $(Id_{V_1}\times Q_{1,n-1})\subset Q_{2,n}$, we may assume
that $X_3=X_{\bolda}$ for some fixed  $\bolda\in \mathcal P'(n-1)$
(Lemma \ref{classp1n}). 
Let $Q'\subset Q_{2,n}$
be the subgroup of elements stabilizing this part
$X_3=X_{\bolda}$, that is
$Q'={\left\{\left.q=\decp{q_1}{q_2}{q_3}\right|q_3\in
    Q_{1,n}^{X_{\bolda}}\right\}}$. For $q\in Q'$ we get
(cf. \eqref{paction}): $$q\cdot X=\decp{0}{q_1X_2q_3^{-1}+q_2
  X_{\bolda}q_3^{-1}}{X_{\bolda}}=\decp{0}{X_2q_1q_3^{-1}+q_2
  q_3^{-1}X_{\bolda}}{X_{\bolda}}.$$ 
Hence, we are reduced to classify the different $Q'$-orbits in $\Hom(U_1,V_1)\stackrel{v.s.}{\cong} \langle {}^tg^i_j \rangle_{i,j}\cong \K^{n-1}$ with respect to the action of $Q'$ given by $$q\cdot X_2=X_2q_1q_3^{-1}+q_2 q_3^{-1}X_{\bolda}.$$

In particular, $Q'\cdot X_2=X_2 \K^*Q_{1,n-1}^{X_{\bolda}}+(\K^{n-1})X_{\bolda}
=X_2 Q_{1,n-1}^{X_{\bolda}}+\Img({}^t X_{\bolda})$. We have $\Img({}^t
X_{\bolda})=\langle {}^tg^i_j\mid j\geqslant 2 \rangle$ and this
subspace is stable under the right action of
$Q_{1,n-1}^{X_{\bolda}}$. There remains to understand the
$Q_{1,n-1}^{X_{\bolda}}$-action on the quotient space
$\K^{n-1}/\Img({}^tX_{\bolda})\cong \langle {}^tg^i_1\mid
i\in[\![1,d_{\bolda}]\!] \rangle$. Under notation of section
\ref{tech}, this corresponds to the right action of
$(Q_{1,n-1}^{X_{\bolda}})_{ext}$ on $\mathcal W:=\langle {}^tg^i_1\mid
i\in[\![1,d_{\bolda}]\!] \rangle.$  
In the left action setting on $ \langle g^i_1\mid
i\in[\![1,d_{\bolda}]\!] \rangle$, $(Q_{1,n-1}^{X_{\bolda}})_{ext}$
can be described as the subgroup stabilizing $\langle g^1_1\rangle$ in
the parabolic subgroup stabilizing each $W_{\ell}=\langle
g_1^i|\lambda_i\geqslant \ell \rangle$ (Lemma \ref{extpara}). 

Picturally, this corresponds to a group of the following form: 
$$ (Q_{1,n-1}^{X_{\bolda}})_{ext}={\tiny\left(\begin{array}{c c c c c c c c c c}*&0&\dots&\dots &0&*&*&*&*&*\\
0 & *&*&*&*&\vdots &\vdots &\vdots &\vdots &\vdots\\
\vdots & *&*&\vdots&\vdots&\vdots &\vdots &\vdots &\vdots &\vdots\\
\vdots & 0&0&*&*&\vdots &\vdots &\vdots &\vdots &\vdots\\
\vdots & \vdots&\vdots&*&*&\vdots &\vdots &\vdots &\vdots &\vdots\\
\vdots & \vdots&\vdots&0  &0 &*&*&*&\vdots &\vdots\\
\vdots & \vdots&\vdots&\vdots  &\vdots &*&*&*&\vdots &\vdots\\
\vdots & \vdots&\vdots&\vdots &\vdots &*&*&*&\vdots &\vdots\\
\vdots & \vdots&\vdots&\vdots &\vdots &0&0&0&* &\vdots\\
\undermat{i=1}{\;\;0\;\;} & \undermat{\{i|\lambda_i>\lambda_1\}}{0&\;\;0\;\;&\;\;0\;\; &0}& \undermat{\{i|\lambda_i\leqslant\lambda_1\}}{0&0&0&0 &*}\\
\end{array}\right).
}$$
\medskip

In the right action setting, let 
$\mathcal W_{\ell}:=\langle {}^tg_1^i\,|\,\lambda_i \leqslant \ell, i\neq 1  \rangle$.
We see that $(Q_{1,n-1}^{X_{\bolda}})_{ext}$ is the subgroup of
$M_{d_{\bolda}}$ stabilizing $\K \,{}^tg^1_1\oplus \mathcal W_{\lambda_1}$ and 
each $\mathcal W_{\ell}$ ($\ell\in\N^*$).

Let $i_0:=\min(\{i>1\, | \, X_2(g_1^{i'})\neq 0 \mbox{ for some $i'>1$ such that $\lambda_i=\lambda_{i'}$}\}\cup\{d_{\bolda}+1\})$ and, if $i_0=d_{\bolda}+1$, we let ${}^tg_1^{i_0}:=0$. We get
\begin{equation*}X_2\cdot (Q_{1,n-1}^{X_{\bolda}})_{ext}\supseteq 
\left\langle{}^tg^i_1\left|\begin{array}{c}i>1\\ \lambda_i=\lambda_{i_0}\end{array}\right.\right\rangle\setminus\{0\}+
\left\langle{}^tg^i_1\left|\begin{array}{c}i>1\\ \lambda_{i}<\lambda_{i_0}\end{array}\right.\right\rangle=:A. \end{equation*}


On the other hand, if $X_2(g^1_1)\neq 0$, we set $\epsilon=1$; otherwise we set $\epsilon=0$. Then:
\begin{equation}\label{acQ1nm1}X_2\cdot (Q_{1,n-1}^{X_{\bolda}})_{ext}=\left\{\begin{array}{c l} A+\K^*({}^tg^1_1)+\left\langle {}^tg^i_1\left| \begin{array}{c} i>1,\\ \lambda_i \leqslant \lambda_1\end{array}\right.\right\rangle 
&\textrm{ if $\epsilon=1$,} \\
A &\textrm{ if $\epsilon=0$.}
\end{array}\right.\end{equation}

Hence, if $\epsilon=1$ and $\lambda_{i_0}\leqslant \lambda_1$, we have $X\in Q_{2,n}\cdot X_{\boldmu}$ with $\boldmu:=(\bolda,0,1)$.
Else, we have $X\in Q_{2,n}\cdot X_{\boldmu}$ with $\boldmu:=(\bolda,\lambda_{i_0},\epsilon)$.

Moreover, given $\bolda\in \mathcal P'(n)$, different elements $(\bolda,l,\epsilon), (\bolda,l',\epsilon') \in \mathcal P''(n)$ give rise to different $(Q_{1,n-1}^{X_{\bolda}})_{ext}$-orbits thanks to \eqref{acQ1nm1}. So if $\boldmu\neq\boldmu'$, we have $Q_{2,n}\cdot X_{\boldmu}\neq Q_{2,n}\cdot X_{\boldmu'}$.
%
%
%
%
%
\end{proof}

Note that we may have $P_{2,n}\cdot X_{\boldmu}=P_{2,n}\cdot X_{\boldmu'}$ with $\boldmu\neq\boldmu'$. A full classification of nilpotent orbits should throw away those cases. However, the description of Lemma \ref{Xmu} will be sufficient for our purpose.

If $\boldmu=(\bolda, \epsilon, l)\in \Pn''(n)$, we denote by $d_{\boldmu}$ the number of parts in the partition of $n$ associated to $\GL_n\cdot X_{\boldmu}$. That is $$d_{\boldmu}=\left\{\begin{array}{l l}d_{\bolda}+1&\textrm{ if $\epsilon=0$ and $l=0$},\\ d_{\bolda} &\textrm{ else.}\end{array}\right.$$

It follows from Lemma \ref{Xmu}  that 
\begin{equation*}
\MN(\pp_{2,n}):=\bigcup_{\boldmu\in \mathcal P''(n)}
\MN_{\boldmu}(\pp_{2,n}), \quad \mbox{ where }\quad
\MN_{\boldmu}(\pp_{2,n})=P_{2,n}\cdot(X_{\boldmu},(\pp^{X_{\boldmu}}_{2,n})^{nil}),
\end{equation*}
\begin{equation*}
\MN(\qq_{2,n}):=\bigsqcup_{\boldmu\in \mathcal P''(n)}
\MN_{\boldmu}(\qq_{2,n}), \quad \mbox{ with }\quad
\MN_{\boldmu}(\qq_{2,n})=Q_{2,n}\cdot(X_{\boldmu},(\qq^{X_{\boldmu}}_{2,n})^{nil}).
\end{equation*}

\begin{lm}\label{dimCmu}
Let $\wfr=\qq_{2,n}$ or $\pp_{2,n}$ and $\boldmu=(\bolda, \epsilon, l)\in \Pn''(n)$.
\begin{enumerate}
\item $(\wfr^{X_{\boldmu}})^{nil}$ is an irreducible subvariety of $\wfr^{X_{\boldmu}}$ of codimension $$c_{\boldmu}=\alter{d_{\boldmu}-1} {\mbox{if $\varepsilon=1$ and $l>0$,}}{d_{\boldmu}}{\mbox{else.}}$$
\item $\overline{\MN_{\boldmu}(\wfr)}$ is a closed irreducible subvariety of $\MN(\wfr)$ of dimension $\dim \wfr-c_{\boldmu}$ 
\end{enumerate}
\end{lm}
\begin{proof}
The computation \eqref{codimnil} remains valid when one replaces $\pp(=\pp_{1,n})$ by $\pp_{2,n}$ or $\qq_{2,n}$. Hence, the second assertion is a consequence of the first one.

The proof is based on case by case considerations on $(\wfr^{X_{\mu}})^{nil}\red$ and the use of Lemma \ref{wfrnil} (or Remark \ref{reml1l2}) in a similar manner as in Lemma \ref{irrpxbolda}.

Firstly, assume that $\varepsilon=0$ or $l=0$. The proof of Lemma \ref{irrpxbolda} can easily be translated here. An elementary base change $(f^i_j)_{i,j}$ based on a reordering of $(e_1, (g^i_j)_{i,j})$ transforms $X_{\boldmu}$ into an element in Jordan canonical form in $M_n$ with partition $\boldmu'\in \mathcal P(n)$.
In these cases, $(\wfr^{X_{\boldmu}})\red$ is defined in $(M_{d_{\boldmu}}^{X_{\boldmu}})\red$ by a condition of one of the types given in the RHS below, for some $i_0$ and possibly $i_1$.
$$ Y\in (\wfr^{X_{\boldmu}})\red\Leftrightarrow 
\;\textrm{(or)}\left\{\begin{array}{l l} Y\red(f^{i_0}_1)\in \K f^{i_0}_1&\scriptsize{(\epsilon=1,l=0)}\\ 
Y\red(f^{i_0}_1)\in \K f^{i_0}_1\textrm{ and }Y\red(f^{i_1}_1)\in \K f^{i_1}_1 &\scriptsize{\left(\begin{array}{c}\epsilon=0, l\neq\lambda_1\\ l+1=\mu'_{i_0}\neq\mu'_{i_1}=\lambda_1+1\end{array}\right)}\\
Y\red(f^{i_0}_1), Y\red(f^{i_1}_1) \in \langle f^{i_0}_1,f^{i_1}_1\rangle& \scriptsize{\left(\begin{array}{c}\wfr=\pp_{2,n}, \epsilon=0,  l=\lambda_1,\\ \mu_{i_0}'=\mu_{i_1}'=l+1\end{array}\right)}\\
 Y\red(f^{i_0}_1)\in \K f^{i_0}_1 \textrm{ and } Y\red(f^{i_1}_1) \in \langle f^{i_0}_1,f^{i_1}_1\rangle& \scriptsize{\left(\begin{array}{c}\wfr=\qq_{2,n}, \epsilon=0,  l=\lambda_1,\\ \mu_{i_0}'=\mu_{i_1}'=l+1\end{array}\right)}\\
\end{array}\right. $$

In particular, $(\wfr^{X_{\boldmu}})\red=\prod_{\ell}\wfr^{X_{\boldmu}}(\ell)$ and each $\wfr^{X_{\boldmu}}(\ell)$ is isomorphic to $M_{\tau_{\ell}}$, $\pp_{1,\tau_{\ell}}$, $\pp_{2,\tau_{\ell}}$ or $\qq_{2,\tau_{\ell}}$.
We then finish as in Lemma \ref{irrpxbolda}.

If $\varepsilon=1$ and $l>0$, we have a more subtle base change to
operate. Let $i_0=\max\{i |\lambda_i>\lambda_1\}$. Recall that the
condition $\epsilon=1$ implies the inequality
$1<i_{\boldmu}\leqslant i_0$ (cf. \eqref{defpseconde}). Let
\begin{equation}\label{basechange}f^i_j:=\left\{\begin{array}{l l}
g_j^{i+1}&\mbox{ if $i<i_0$, $i+1\neq i_{\boldmu}$ and $1\leqslant j\leqslant \lambda_{i+1}$},\\
g_{j-1}^{i+1}&\mbox{ if $i+1=i_{\boldmu}$ and $1<j\leqslant \lambda_{i_{\boldmu}}+1$},\\
e_1&\mbox{ if $i+1=i_{\boldmu}$ and $j=1$},\\
g_{j}^{1}-g_j^{i_{\boldmu}}&\mbox{ if $i=i_{0}$ and $1\leqslant j\leqslant \lambda_{1}$}.\\
g_{j}^{i}&\mbox{ if $i>i_{0}$ and $1\leqslant j\leqslant \lambda_{i}$},\\
\end{array}\right.\qquad 
\end{equation}
In this new basis, $X_{\boldmu}$ is in Jordan canonical form associated to a partition $\boldmu'=(\mu'_1\geqslant \dots\geqslant \mu'_{d_{\bolda}})\in \Pn(n)$ and $\qq_{2,n}$ (resp. $\pp_{2,n}$) is characterized by the two conditions
\begin{equation}\label{caracp2n}Y\in \qq_{2,n} \textrm{ (resp. }\pp_{2,n})\Leftrightarrow \left\{\begin{array}{l}Y(f_1^{i_{\boldmu}-1})\in \K f_1^{i_{\boldmu}-1} \quad
(\mbox{resp. } Y(f_1^{i_{\boldmu}-1})\in 
\langle f_1^{i_{\boldmu}-1},  f_1^{i_{0}}+ f_2^{i_{\boldmu}-1}\rangle), \\
Y( f_1^{i_{0}}+ f_2^{i_{\boldmu}-1})\in 
\langle f_1^{i_{\boldmu}-1},  f_1^{i_{0}}+ f_2^{i_{\boldmu}-1}\rangle.
\end{array}\right.
\end{equation}
Define $\ell_1:=\mu'_{i_{\boldmu}-1}=\lambda_{i_{\boldmu}}+1=l+1$ and  $\ell_2:=\mu'_{i_0}=\lambda_1$.
From now on, we assume that $Y\in M_n^{X_{\boldmu}}$. Then $Y(f_1^{i_{\boldmu}-1})$ has no component in $f_2^{i_{\boldmu}-1}$ (Lemma \ref{TA}). Hence, for such $Y$, the two conditions on the first line of \eqref{caracp2n} are both equivalent to the existence of some $\alpha\in \K $ such that $Y(f_1^{i_{\boldmu}-1})=\alpha f_1^{i_{\boldmu}-1}$. 

Now, write $Y(f_1^{i_0})=\sum_i \beta_i f_1^i$  and $Y(f_2^{i_{\boldmu}-1})=\sum_i\gamma_i f_1^i+\gamma'_i f_2^i$ (Lemma \ref{TA}). We note that $\gamma'_{i_{\boldmu}-1}=\alpha$ and, since $\mu'_{i_{\boldmu}-1}= \lambda_{i_{\boldmu}}+1\geqslant \lambda_1+2=\mu'_{i_0}+2$, we have $\gamma_{i}=0$ for all $i$ such that $\mu'_i=\mu'_{i_0}$ (Lemma \ref{TA}).
Hence the second condition of \eqref{caracp2n}, $Y( f_1^{i_{0}}+ f_2^{i_{\boldmu}-1})=\xi f_1^{i_{\boldmu}-1}+\delta(f_1^{i_{0}}+ f_2^{i_{\boldmu}-1})$, implies $\beta_{i_0}=\delta=\gamma'_{i_{\boldmu}-1}=\alpha$ and $\beta_{i}=0$ for all $i\neq i_{0}$ such that $\mu'_{i}=\mu'_{i_0}$.
Thus,  we have the following characterization of $(\wfr^{X_{\boldmu}})\red$ in $M_n^{X_{\boldmu}}$:
$$Y\red\in (\wfr^{X_{\boldmu}})\red\Leftrightarrow Y(\ell_1),Y(\ell_2)=\left(\decpp{\alpha}{A_1}{B_1},\decpp{\alpha}{A_2}{B_2}\right), \quad
\begin{array}{c}\alpha\in \K,\\ A_j\in M_{1,\tau_{\ell_j}-1},\\ B_j\in M_{\tau_{\ell_j}-1}.\end{array}$$
Hence $\wfr^{X_{\bolda}}\red=\wfr^{X_{\bolda}}(\ell_1, \ell_2)\times\prod_{\ell\notin\{\ell_1,\ell_2\}} \wfr^{X_{\bolda}}(\ell)$; $\wfr^{X_{\boldmu}}(\ell)=M_{\tau_{\ell}}$ for $\ell\neq
\ell_1,\ell_2$ and $(\wfr^{X_{\boldmu}}(\ell_1,\ell_2))^{nil}$ is
characterized in $\wfr^{X_{\boldmu}}(\ell_1,\ell_2)$ by the conditions
$\alpha=0$, $B_1, B_2$ nilpotent (Lemma \ref{commons}). Thus
$(\wfr^{X_{\boldmu}}(\ell_1,\ell_2))^{nil}$ is an irreducible variety
of codimension $\tau_{\ell_1}+\tau_{\ell_2}-1$ in
$\wfr^{X_{\boldmu}}(\ell_1,\ell_2)$ (Lemma \ref{commons}); the variety
$(\wfr^{X_{\boldmu}})^{nil}$ is also irreducible and
$\codim_{\wfr^{X_{\boldmu}}}(\wfr^{X_{\boldmu}})^{nil}=d_{\boldmu}-1$
(Remark \ref{reml1l2}). Hence we have proved the first assertion follows in this last case. 
\end{proof}

\begin{thm}\label{propp2n}
Let $\wfr=\qq_{2,n}$ or $\pp_{2,n}$. Then $\MN(\wfr)$ is equidimensional of dimension $\dim \wfr-1$. It 
 has $\left \lfloor \frac n 2\right \rfloor$ components. 
\end{thm}
\begin{proof}
We have $\min\{c_{\boldmu}|\boldmu\in \mathcal P''(n)\}=1$. Hence, it follows from Lemma \ref{dimCmu} and Proposition \ref{dimqq} that each irreducible component of $\MN(\qq_{2,n})$ has dimension $\dim \qq_{2,n}-1$. There are two types of $\boldmu\in \Pn''(n)$ such that $c_{\boldmu}=1$. 
\begin{itemize}
\item $\boldmu=((n-1,\emptyset),0,1)$ which is the only element whose associated partition of $n$ has just one part.
\item $\boldmu=((\lambda_1,\lambda_2),\lambda_2,1)$ with $\lambda_2>\lambda_1$. Its associated partition of $n$ has two parts: $(\lambda_2+1\geqslant \lambda_1)$, cf.  \eqref{basechange} for more details. Note that this covers (the transpose of) the partitions involved in the proof of Proposition \ref{composantesDeDimNMoins1DansSdeuxno} since $\lambda_2>\lambda_1\Leftrightarrow(\lambda_2+1)-\lambda_1\geqslant 2$.
\end{itemize} There are $\left \lfloor \frac n 2\right \rfloor$ such elements, whence the statement for $\wfr=\qq_{2,n}$.

It follows from the description above that the map $\{\boldmu\in
\mathcal P''(n)| c_{\boldmu}=1\}\rightarrow \mathcal P(n)$ which sends
$\boldmu$ to the partition associated to $\GL_n.X_{\boldmu}$ is
injective. In particular, the different such $X_{\boldmu}$ belong to
different $P_{2,n}$-orbits. So the associated varieties
$\overline{\MN_{\boldmu}(\pp_{2,n})}$, which are the irreducible
components of maximal dimension of $\MN(\pp_{2,n})$, are all distinct.

There remains to prove that there is no other irreducible component in $\MN(\pp_{2,n})$. Let 
$(X,Y)\in \MN(\pp_{2,n})$. The pair $(X_{|V_2}, Y_{|V_2})$ is a commuting pair in $\gl(V_2)$ hence, up to $\GL(V_2)\times Id_{U_2}(\subset P_{2,n})$-conjugacy, we can assume that $X(e_1)=Y(e_1)=0$. That is $(X,Y)\in \MN(\qq_{2,n})$. In particular, there exists $\boldmu\in \mathcal P''(n)$ such that $(X,Y)\in \overline{\MN_{\boldmu}(\qq_{2,n})}\subset \overline{\MN_{\boldmu}(\pp_{2,n})}$ and $c_{\boldmu}=1$. We have therefore shown that $$\MN(\pp_{2,n})\subset \bigcup_{c_{\boldmu}=1}\overline{\MN_{\boldmu}(\pp_{2,n})},$$
and we are done.
\end{proof}

\begin{rem} \label{remfin}(i) The key point of this last proof in the case $\wfr=\pp_{2,n}$ is that $\dim \MN_{\boldmu}(\qq_{2,n})$ and $\dim \MN_{\boldmu}(\pp_{2,n})$ are both related to the same integer $c_{\boldmu}$. This is what allows us to carry out the equidimensionality property from $\MN(\qq_{2,n})$ to $\MN(\pp_{2,n})$\\
(ii) The method used in this section is deeply based on the decomposition of $\MN(\wfr)$ as a finite union of the irreducible subvarieties $\MN_{\boldmu}(\wfr)$. 
For this, the classification into finitely many orbits of Lemma \ref{Xmu} plays a key role. This situation breaks down in general for $\pp_{k,n}$. Using quiver theory and techniques similar to \cite{Bo}, M. Reineke communicated to us an example of an infinite family of $P_{6,12}$-orbits in $\pp_{6,12}$. \\
(iii)  Similarly, in \cite{GR}, the authors show that some continuous families of $Q_{n,n}$-orbits exist in $\qq_{n,n}$ (Borel case) as soon as $n\geqslant 6$. From this, they deduce the existence of irreducible components of $\MN(\qq_{n,n})$ of dimension greater or equal than $\dim \qq_{n,n}$ showing that the variety is not equidimensional in these cases.
\end{rem}

\begin{coro}\label{corpropp2n}
  $\sdeuxno$,$\sxno{n-2}$, $\sxnoemb{n-2}$ are equidimensional of
  dimension $n-1$. They have $\left \lfloor \frac{n}{2} \right \rfloor$ components.
\end{coro}
\begin{proof}
  The number of components in  $\sdeuxno$ is (Proposition \ref{corresirred}) the number of 
  components in $\MN^{cyc}(\pp_{n-2,n})$, thus at most the number
  $\left \lfloor \frac{n}{2} \right\rfloor$ of components in the variety
  $\MN(\pp_{n-2,n})$ which may contain noncyclic components. On the
  other hand, we have exhibited $\left \lfloor \frac{n}{2} \right \rfloor$
  components of dimension $n-1$ in  $\sdeuxno$ in Proposition 
 \ref{composantesDeDimNMoins1DansSdeuxno}, hence the conclusion for
 $\sdeuxno$. The same argument applies to $\sxno{n-2}$, using Remark~\ref{remCodim2commeDim2}.

Finally, from the existence of a surjective morphism
 $\sxnoemb{n-2}\to  \sxno{n-2}$ (Proposition \ref{schemasintermediaires}), we see that  $\sxnoemb{n-2}$
has at least $ \left \lfloor \frac{n}{2} \right \rfloor$ components. 
But Theorem \ref{propp2n} implies that there are at most $\left \lfloor \frac{n}{2}
\right \rfloor$ components, and that these components have dimension $n-1$. 
The result follows.
\end{proof}
\bibliographystyle{plain}

\end{document}